\documentclass{article}

\usepackage{amsmath,amssymb,bm,bbm,mathrsfs,amscd}
\usepackage{color}
\usepackage{amsthm}
\usepackage{dsfont}
\usepackage{graphicx}
\usepackage{epsfig}
\usepackage{subfigure}
\usepackage{algorithm}
\usepackage[noend]{algpseudocode}
\usepackage{tikz}
\usepackage{enumitem}
\usepackage{stackrel}

\newtheorem{theorem}{Theorem}
\newtheorem{definition}{Definition}

\newtheorem{lemma}{Lemma}
\newtheorem{corollary}{Corollary}
\newtheorem{remark}{Remark}

\newtheorem{assumption}{Assumption}
\newtheorem{standassumption}{Standing Assumption}
\newtheorem{conjecture}{Conjecture}

\newcommand{\RR}{{\mathbb{R}}}
\newcommand{\NN}{{\mathbb{N}}}
\newcommand{\EE}{{\mathbb{E}}}
\newcommand{\PP}{{\mathbb{P}}}
\newcommand{\JJ}{{\mathbb{J}}}
\newcommand{\FF}{{\mathbb{F}}}

\newcommand{\VI}{{\mathrm{VI}}}

\newcommand{\mc}{\mathcal}
\newcommand{\norm}[1]{\left\|#1\right\|}
\newcommand{\normsq}[1]{\left\|#1\right\|^2}
\newcommand{\EEk}[1]{\EE\left[#1|\mc F_k\right]}
\newcommand{\EEx}[1]{\EE\left[#1\right]}
\newcommand{\op}{\operatorname}

\newcommand{\prox}{\operatorname{prox}}

\newcommand{\bs}{\boldsymbol}
\newcommand{\fineass}{\hfill\small$\square$}
\newcommand{\normsqphi}[1]{\left\|#1\right\|_\Phi^2}

\title{\LARGE \bf Distributed Forward--Backward algorithms for stochastic generalized Nash equilibrium seeking}

\author{Barbara Franci and Sergio Grammatico
\thanks{The authors are in the Delft Center for System and Control, TU Delft, The Netherlands. E-mail addresses:
        {\tt\small \{b.franci-1, s.grammatico\}@tudelft.nl}.}
\thanks{This work was partially supported by NWO under research projects OMEGA (613.001.702) and P2P-TALES (647.003.003), and by the ERC under research project COSMOS (802348).}}

\begin{document}

\maketitle

\begin{abstract}
We consider the stochastic generalized Nash equilibrium problem (SGNEP) with expected-value cost functions. 
Inspired by Yi and Pavel (Automatica, 2019), we propose a distributed GNE seeking algorithm based on the preconditioned forward--backward operator splitting for SGNEP, where, at each iteration, the expected value of the pseudogradient is approximated via a number of random samples.

As main contribution, we show almost sure convergence of our proposed algorithm if the pseudogradient mapping is restricted (monotone and) cocoercive. For non--generalized SNEPs, we show almost sure convergence also if the pseudogradient mapping is restricted strictly monotone.

Numerical simulations show that the proposed forward--backward algorithm seems faster that other available algorithms.
\end{abstract}

\paragraph{Keywords.}
Stochastic generalized Nash equilibrium problems, stochastic variational inequalities.


\section{Introduction}

Nash equilibrium problems (NEPs) have been widely studied since their first formulation \cite{nash1950}. In a NEP, a set of agents interacts with the aim of minimizing their cost functions. A number of results are present concerning existence ad uniqueness of an equilibrium as well as methodologies to compute one \cite{li2013,sandholm2010}. 
The interactions between the agents are expressed, in this case, only through the cost function. For the sake of generality and of making the problem realistic, coupling constraints have been considered as well, since the introduction of so-called generalized Nash equilibrium problems (GNEPs) \cite{debreu1952}. This class of games has been recently studied by the multi-agent system and control community \cite{yi2019,kulkarni2012,grammatico2017,pavel2019,salehisadaghiani2019,kanzow2019}. One main reason for this interest is related to possible applications that range from economics, to engineering and operation research \cite{kulkarni2012,pavel2007}.

The characteristic of GNEPs is that each agent seeks to minimize his own cost function under some \textit{joint} feasibility constraints. Namely, both the cost function and the constraints depend on the strategies chosen by the other agents. Consequently, the search for GNE is usually very difficult. 

Similarly to NEPs, a number of results are present concerning algorithms and methodologies to find an equilibrium in a GNEP \cite{facchinei2010,facchinei2007}. 
In the deterministic case, many algorithm are available to find a solution, both distributed or semi-decentralized \cite{yi2019,belgioioso2018,belgioioso2017}. 
Among the possible methods to reach an equilibrium, an elegant approach is to seek for a solution of the associated variational inequality (VI) \cite{facchinei2007}. 

To recast the problem as a VI, the Karush--Kuhn--Tucker (KKT) conditions are considered and the problem is rewritten as a monotone inclusion. Such a problem can then be solved via operator splitting techniques.
Among others, we focus on the forward--backward (FB) splitting which leads to one of the fastest and definitely simplest algorithm available \cite{bau2011}. 

The downside of the FB scheme is that, when applied to GNEPs, it is not distributed. 
When considering game-theoretic setup, it is desirable (and more realistic) to consider distributed algorithms, in the sense that each agent should only know its local cost function and its local constraints. For this reason, \textit{preconditioning} has been recently introduced in \cite{yi2019}.
In \cite{franci19}, we propose a preliminary extension of this method to the stochastic case.

A stochastic NEP (SNEP) is a NEP where the cost functions are expected value functions. Such problems arise when there is some uncertainty, expressed through a random variable with an unknown distribution. Unfortunately, SNEPs are not studied as much as their deterministic counterpart, despite a number of practical problems must be modelled with uncertainty. Among others, in transportation systems, a possible source of uncertainty is the drivers perception of travel time \cite{watling2006}; in electricity markets, companies produce energy without knowing in advance exactly the demand \cite{henrion2007}. Similarly to the deterministic case, if we also consider shared constraints, the problem can be modelled with stochastic GNEPs (SGNEPs) \cite{xu2013}. For instance, any networked Cournot games with market capacity constraints, with uncertainty in the demand, can be modelled in this case \cite{demiguel2009,abada2013}. Due to their wide applicability, SGNEPs have received quite some attention from the control community as well \cite{ravat2011,koshal2013,yu2017}.


A SGNEP is a GNEP with expected value cost functions. Indeed, if the random variable is known, the expected value formulation can be solved with any standard technique for deterministic variational inequalities. However, the pseudogradient is usually not directly accessible, for instance due to the need of excessive computations in performing the expected value. For this reason, in many situations, the search for a solution of a SVI relies on samples of the random variable.
Essentially, there are two main methodologies available: sample average approximation (SAA) and stochastic approximation (SA). In the SAA approach, we replace the expected value formulation with the average over an increasing number of samples of the random variable. 
This approach is practical when there is a huge number of data available as, for instance, in Monte Carlo simulations or machine learning \cite{staudigl2019, iusem2017}. 
In the SA approach, each agent can sample only one realization of the random variable. This approach is less computationally expensive, but, not surprisingly, it usually requires stronger assumptions on the mappings involved \cite{koshal2013,yousefian2017,yousefian2014}.

One of the very first formulations of a SA approach for a stochastic FB problem was made in \cite{jiang2008}, under the assumption of strong monotonicity and Lipschitz continuity of the mapping involved. 
In \cite{rosasco2016} instead, convergence is proved under cocoercivity \textit{and} uniform monotonicity. 
To weaken the assumptions, algorithms more involved than the FB have been proposed and studied in the literature. For instance, in a recent paper, \cite{staudigl2019}, the authors propose a forward-backward-forward (FBF) algorithm that converges to a solution under the assumption of pseudomonotone pseudo-gradient mapping but it requires two costly evaluations of the pseudogradient mapping. Alternatively, under the same assumptions, one can consider the extragradient (EG) method proposed in \cite{iusem2017} which takes two projection steps that can be slow. Therefore, taking weaker assumptions comes at the price of having computational complexity and slowness of the algorithms.


In this paper, we present a FB algorithm and prove its convergence both for SGNEPs and SNEPs. In particular, our main contributions are the following:
\begin{itemize}
\item We present the first preconditioned FB algorithm for SGNEPs with non smooth cost functions and prove its almost sure convergence to an equilibrium.
\item Almost sure convergence is shown under restricted cocoercivity of the pseudo-gradient mapping by using the SAA scheme (Sections \ref{sec_SGNEPs} and \ref{sec_conv}). 
\item For SNEPs, the FB algorithm converge almost surely under restricted cocoercivity or restricted strict monotonicity also with the SA scheme (Section \ref{sec_SNEPs}).
\end{itemize}
 
Our assumptions are weaker when compared to the current literature on FB algorithms. Indeed, the preconditioning technique is presented in \cite{yi2019,belgioioso2018} under strong monotonicity while FB algorithms for stochastic VIs converge almost surely for cocoercive and uniformly monotone \cite{rosasco2016} or strongly monotone operators \cite{jiang2008, robbins1951}. Moreover, compared to the naively monotone FBF and EG, our algorithm shows faster convergence both in number of iterations and computational time.

We point out that, in both cases, we suppose the pseudogradient to be Lipschitz continuous, but knowing the Lipschitz constant is not necessary. This is remarkable since computing the Lipschitz constant can be challenging in a distributed setup. 

A preliminary study related to this work is submitted for publication \cite{franci19}. In that paper, we considered a SGNEP and build a preconditioned FB algorithm with damping. The algorithm is guaranteed to reach a SGNE if the pseudogradient mapping is strongly monotone and its convergence follows directly from \cite{rosasco2016}. We here motivate in details that the assumption of uniform monotonicity taken in \cite{rosasco2016} can be replaced with assumptions related to specific properties of the selected approximation schemes. Moreover, restricted cocoercivity is enough for the analysis and to ensure convergence. 


\section{Notation and preliminaries}

We use Standing Assumptions to postulate technical conditions that implicitly hold throughout the paper while Assumptions are postulated only when explicitly called.

$\RR$ denotes the set of real numbers and $\bar\RR=\RR\cup\{+\infty\}$.
$\langle\cdot,\cdot\rangle:\RR^n\times\RR^n\to\RR$ denotes the standard inner product and $\norm{\cdot}$ represents the associated Euclidean norm. We indicate a (symmetric and) positive definite matrix $A$, i.e., $x^\top Ax>0$, with $A\succ0$. Given a matrix $\Phi\succ0$, we denote the $\Phi$-induced inner product, $\langle x, y\rangle_{\Phi}=\langle \Phi x, y\rangle$. The associated $\Phi$-induced norm, $\norm{\cdot}_{\Phi}$, is defined as $\norm{x}_{\Phi}=\sqrt{\langle \Phi x, x\rangle}$. 
$A\otimes B$ indicates the Kronecker product between matrices $A$ and $B$. ${\bf{0}}_m$ indicates the vector with $m$ entries all equal to $0$. Given $N$ vectors $x_{1}, \ldots, x_{N} \in \RR^{n}$, $\boldsymbol{x} :=\op{col}\left(x_{1}, \dots, x_{N}\right)=\left[x_{1}^{\top}, \dots, x_{N}^{\top}\right]^{\top}.$

$\op{J}_F=(\op{Id}+F)^{-1}$ is the resolvent of the operator $F:\RR^n\to\RR^n$ where $\op{Id}$ indicates the identity operator. The set of fixed points of $F$ is $\op{fix}(F):=\{x\in\RR^n\mid x\in F(x)\}$. 
For a closed set $C \subseteq \RR^{n},$ the mapping $\op{proj}_{C} : \RR^{n} \to C$ denotes the projection onto $C$, i.e., $\op{proj}_{C}(x)=\op{argmin}_{y \in C}\|y-x\|$. The residual mapping is, in general, defined as $\op{res}(x^k)=\norm{x^k-\op{proj}_{C}(x^k-F(x^k))}.$ Given a proper, lower semi-continuous, convex function $g$, the subdifferential is the operator $\partial g(x):=\{u\in\Omega \mid (\forall y\in\Omega):\langle y-x,u\rangle+g(x)\leq g(y)\}$. The proximal operator is defined as $\prox_{g}(v):=\op{argmin}_{u\in\Omega}\{g(u)+\frac{1}{2}\norm{u-v}^{2}_{}\}=\mathrm{J}_{\partial g}(v)$.
$\iota_C$ is the indicator function of the set C, i.e., $\iota_C(x)=1$ if $x\in C$ and $\iota_C(x)=0$ otherwise. The set-valued mapping $\mathrm{N}_{C} : \RR^{n} \to \RR^{n}$ denotes the normal cone operator for the the set $C$ , i.e., $\mathrm{N}_{C}(x)=\varnothing$ if $x \notin C,\left\{v \in \RR^{n} | \sup _{z \in C} v^{\top}(z-x) \leq 0\right\}$ otherwise.

We now recall some basic properties of operators \cite{facchinei2007}. First, we recall that $F$ is $\ell$-Lipschitz continuous if, for some $\ell>0$
$$
\norm{F(x)-F(y)} \leq \ell\norm{x-y} \text { for all } x, y \in \RR^{n}.
$$
\begin{definition}[Monotone operators]
A mapping $F:\op{dom}F\subseteq\RR^n\to\RR^n$ is:
\begin{itemize}
\item (strictly) monotone if for all $x, y \in \op{dom}(F)$ $(x\neq y)$
$$
\langle F(x)-F(y),x-y\rangle\geq\,(>)\,0;
$$
\item $\beta$-cocoercive with $\beta>0$, if for all $x, y \in \op{dom}(F)$
$$
\langle F(x)-F(y),x-y\rangle \geq \beta\|F(x)-F(y)\|^{2};
$$
\item firmly non expansive if for all $x, y \in \op{dom}(F)$
$$
\|F(x)-F(y)\|^{2} \leq\|x-y\|^{2}-\|(\mathrm{Id}-F) (x)-(\mathrm{Id}-F) (y)\|^{2}.
$$
\end{itemize}
We use the adjective ``restricted'' if a property holds for all $(x,y)\in\op{dom}(F)\times \op{fix}(F)$.
\fineass\end{definition}
An example of firmly non expansive operator is the projection operator over a nonempty, compact and convex set \cite[Proposition 4.16]{bau2011}. We note that a firmly non expansive operator is also cocoercive, hence monotone and non expansive \cite[Definition 4.1]{bau2011}.

%


\section{Mathematical background: Generalized Nash equilibrium problems}\label{sec_GNEPs}
We consider a set $\mc I=\{1,\dots,N\}$ of noncooperative agents, each of them choosing its strategy $x_i\in\RR^{n_i}$ from its local decision set $\Omega_i\subseteq\RR^{n_i}$. The aim of each agent is to minimize its local cost function within its feasible strategy set. We call $\bs{x}_{-i}=\op{col}(\{x_j\}_{j \neq i})$ the decisions of all the agents with the exception of $i$ and set $n=\sum_{i=1}^N n_i$. 
The local cost function of agent $i$ is a function $J_i(x_i,\bs x_{-i}):\RR^{n_i}\times\RR^{n-n_i}\to\RR$ that depends on both the local variable $x_i$ and the decision of the other agents $\bs x_{-i}$. The cost function has the form
\begin{equation}\label{eq:f}
J_{i}(x_{i}, \bs x_{-i}):=f_{i}(x_{i},\bs x_{-i}) + g_{i}(x_{i}),
\end{equation}
that present the typical splitting in smooth and non smooth parts. We assume that the non smooth part is represented by $g_i:\RR^{n_i}\to\bar\RR$. We note that it can model not only a local cost, but also local constraints via an indicator function, e.g. $g_{i}(x_{i})=\iota_{\Omega_i}(x_i)$.
\begin{standassumption}[Local cost]\label{ass_G}
For each $i\in\mc I$, the function $g_i$ in \eqref{eq:f} is lower semicontinuous and convex and $\op{dom}(g_{i})=\Omega_i$ is nonempty, compact and convex.
\fineass\end{standassumption}



\begin{standassumption}[Cost functions convexity]\label{ass_J_det}
For each $i \in \mc I$ and $\boldsymbol{x}_{-i} \in \mc{X}_{-i}$ the function $f_{i}(\cdot, \boldsymbol{x}_{-i})$ is convex and continuously differentiable.
\fineass\end{standassumption}

Furthermore, we consider a game with affine shared constraints $A\bs x\leq b$. Therefore, the feasible decision set of each agent $i \in \mc I$ is denoted by the set-valued mapping $\mc X_i,$:
\begin{equation}\label{constr}
\mc{X}_i(\bs{x}_{-i}) :=\textstyle\left\{y_i \in \Omega_i\; |\; A_i y_i \leq b-\sum_{j \neq i}^{N} A_j x_j\right\},
\end{equation}
where $A_i \in \RR^{m \times n}$ and $b\in\RR^m$. The set $\Omega_i$ represents the local decision set for agent $i$, while the matrix $A_i$ defines how agent $i$ is involved in the coupling constraints. The collective feasible set can be written as
\begin{equation}\label{collective_set}
\bs{\mc{X}}=\left\{\bs y \in\bs\Omega\; | \;A\bs y-b \leq {\bf{0}}_{m}\right\}
\end{equation}
where $\bs\Omega=\prod_{i=1}^N\Omega_i$ and $A=\left[A_{1}, \ldots, A_{N}\right]\in\RR^{m\times n}$. 


\begin{standassumption}[Constraints qualification]\label{ass_X}
The set $\bs{\mc{X}}$ satisfies Slater's constraint qualification. 
\fineass\end{standassumption}

The aim of each agent $i$, given the decision variables of the other agents $\bs{x}_{-i}$, is to choose a strategy $x_i$, that solves its local optimization problem, i.e.,
\begin{equation}\label{game_det}
\forall i \in \mc I: \quad\left\{\begin{array}{cl}
\min\limits _{x_i \in \Omega_i} & J_i\left(x_i, \bs{x}_{-i}\right) \\ 
\text { s.t. } & A_i x_i \leq b-\sum_{j \neq i}^{N} A_j x_j.
\end{array} \right.
\end{equation}
When the optimization problems are simultaneously solved, the solution concept that we are seeking is that of generalized Nash equilibrium.
\begin{definition}\label{def_GNE}
A Generalized Nash equilibrium (GNE) is a collective strategy $\bs x^*\in\bs{\mc X}$ such that, for all $i \in \mc I$,
$$J_i(x_i^{*}, \boldsymbol x_{-i}^{*}) \leq \inf \{J_i(y, \boldsymbol x_{-i}^{*})\; | \; y \in \mc{X}_i(\bs x_{-i})\}\vspace{-.55cm}.$$
\fineass\end{definition}

In other words, a GNE is a set of strategies where no agent can decrease its objective function by unilaterally deviating from its decision.
While, under Assumptions \ref{ass_G}--\ref{ass_X}, existence of GNE of the game is guaranteed by \cite[Section 4.1]{facchinei2010}, uniqueness does not hold in general \cite[Section 4.3]{facchinei2010}. 

Within all the possible Nash equilibria, we focus on those that corresponds to the solution set of an appropriate variational inequality. 
To this aim, let the (pseudo) gradient of the cost function be
\begin{equation}\label{eq_sub}
F(\bs{x})=\op{col}\left(\nabla_{x_{1}} f_{1}\left(x_{1}, \bs{x}_{-1}\right), \dots, \nabla_{x_{N}} f_{N}\left(x_{N}, \bs{x}_{-N}\right)\right),
\end{equation}
and let
\begin{equation}\label{eq_G}
G(\bs x)=\op{col}(\partial g_1(x_1),\dots,\partial g_N(x_N)).
\end{equation}
Formally, the \textit{variational inequality} problem, $\op{VI}(\bs{\mc X},F,G)$, is \cite[Definition 26.19]{bau2011}: find $\bs x^{*} \in \bs{\mc X}$ such that
\begin{equation}\label{eq_vi}
\langle F(\bs x^*),\bs x-\bs x^*\rangle+G(\bs x)-G(\bs x^*)\geq 0\text { for all } \bs x \in \bs{\mc X}.
\end{equation}
\begin{remark}
Under Standing Assumptions \ref{ass_G}--\ref{ass_X}, the solution set of $\op{VI}(\bs{\mc X},F,G)$ is non empty and compact, i.e. $\op{SOL}(\bs{\mc X},F,G)\neq\varnothing$ \cite[Corollary 2.2.5]{facchinei2007}.
\end{remark}

When Standing Assumptions \ref{ass_G}--\ref{ass_X} hold, any solution of $\op{VI}(\bs{\mc X},F,G)$ in \eqref{eq_vi} is a GNE of the game in (\ref{game_det}), while vice versa does not hold in general. Indeed, there may be Nash equilibria that are not solution of the VI \cite[Proposition 12.7]{palomar2010}. The GNE that are also solution of the associated VI are called variational equilibria (v-GNE).


The GNEP can be recasted as a monotone inclusion, namely, the problem of finding a zero of a set-valued monotone operator. To this aim, we characterize a GNE of the game in terms of the Karush--Kuhn--Tucker (KKT) conditions of the coupled optimization problems in (\ref{game_det}). 
Then, the set of strategies $\bs x^{*}$ is a GNE if and only if the corresponding KKT conditions are satisfied \cite[Theorem 4.6]{facchinei2010}. Moreover, since the variational problem is a minimization problem, appropriate KKT conditions hold also in this case. The interest in such conditions is due to \cite[Theorem 3.1]{facchinei2007vi} that provides a criteria to select GNE that are also solution of the VI. Specifically, under assumptions \ref{ass_G}--\ref{ass_X}, \cite[Theorem 3.1]{facchinei2007vi} establishes that the equilibrium points are those such that the shared constraints have the same dual variable for all the agents. We refer also to \cite[Theorem 3.1]{auslender2000} for a more general result.

%
%

\begin{remark}\label{remark_NE}
For non-generalized Nash equilibrium problems, being a solution of the VI is necessary and sufficient to be a Nash equilibrium. Formally, a collective strategy $\bs x^*\in\mc X$ is a Nash equilibrium of the game if and only if $\bs x^*$ satisfies \eqref{eq_vi} \cite[Proposition 1.4.2]{facchinei2007}.
\fineass\end{remark}

In \cite{yi2019}, the authors propose a distributed preconditioned FB algorithm as in Algorithm \ref{algo_FB_det_i}. Inspired by this work, in the following sections we describe SGNEPs and propose the stochastic counterpart of Algorithm \ref{algo_FB_det_i} (Algorithm \ref{algo_FB_stoc_i}). Therefore, more details on the preconditioning procedure and on how to obtain the algorithm are described in Section \ref{sec_algo_SGNEP}.
\begin{remark}
When the local cost function is the indicator function, we can use the projection on the local feasible set $\Omega_i$, instead of the proximal operator \cite[Example 12.25]{bau2011}.
\fineass\end{remark}

\begin{algorithm}[h]
\caption{Preconditioned Forward Backward \cite{yi2019}}\label{algo_FB_det_i}
Initialization: $x_i^0 \in \Omega_i, \lambda_i^0 \in \RR_{\geq0}^{m},$ and $z_i^0 \in \RR^{m} .$\\
Iteration $k$: Agent $i$\\
($1$) Receives $x_j^k$ for $j \in \mathcal{N}_{i}^{J}, \lambda_j^k$ for $j \in \mathcal{N}_{i}^{\lambda}$, then updates
$$\begin{aligned}
&x_i^{k+1}=\op{prox}_{g_i}[x_i^k-\alpha_i(\nabla_{x_{i}} J_{i}(x_i^k,\boldsymbol x_{-i}^k)-A_{i}^{\top} \lambda_i^k)]\\
&z_i^{k+1}=z_i^k-\nu_{i} \sum_{j \in \mathcal{N}_{i}^{\lambda}} w_{i,j}(\lambda_i^k-\lambda_j^k)
\end{aligned}$$
($2$) Receives $z_{j, k+1}$ for $j \in \mathcal{N}_{i}^{\lambda}$, then updates
$$\begin{aligned}
 \lambda_i^{k+1}=&\op{proj}_{\RR^m_{\geq 0}}\{\lambda_i^k-\sigma_i[A_{i}(2 x_i^{k+1}-x_i^k)-b_{i}\\
&+\sigma_i\sum_{j \in \mathcal{N}_{i}^{\lambda}} w_{i,j}[2(z_i^{k+1}-z_{j, k+1})-(z_i^k-z_j^k)]\\
&+\sigma_i\sum_{j \in \mathcal{N}_{i}^{\lambda}} w_{i,j}(\lambda_i^k-\lambda_j^k)]\}\\
\end{aligned}$$
\end{algorithm}

\section{Stochastic generalized Nash equilibrium problems}\label{sec_SGNEPs}
\subsection{Stochastic equilibrium problem formulation}
In this section we describe the stochastic counterpart of GNEPs (SGNEPs). With SGNEP we mean a GNEP where the cost function is an expected value function. As in the deterministic case, we consider a set of agents $\mc I=\{1,\dots,N\}$ with strategies $x_i\in\Omega_i\subseteq\RR^{n_i}$. Each of the agents seek to minimize its local cost function within its feasible strategy set $\Omega_i$ that satisfy Standing Assumption \ref{ass_G}.
The local cost function of agent $i$ is defined as 
\begin{equation}\label{eq_cost_stoc}
\JJ_i(x_i,\bs{x}_{-i}):=\EE_\xi[f_i(x_i,\bs{x}_{-i},\xi(\omega))] + g_{i}(x_{i}).
\end{equation}
for some measurable function $f_i:\mc \RR^{n}\times \RR^d\to \RR$. We suppose that the cost functions in \eqref{eq_cost_stoc} satisfy Standing Assumptions \ref{ass_G} and \ref{ass_J_det}. The uncertainty is expressed through the random variable $\xi:\Xi\to\RR^d$ where $(\Xi, \mc F, \PP)$ is the probability space. The cost function depends on both the local variable $x_i$, the decision of the other agents $\bs x_{-i}$ and on the random variable $\xi$. $\EE_\xi$ represent the mathematical expectation with respect to the distribution of the random variable $\xi(\omega)$\footnote{From now on, we use $\xi$ instead of $\xi(\omega)$ and $\EE$ instead of $\EE_\xi$.}. We assume that $\EE[f_i(\bs{x},\xi)]$ is well defined for all the feasible $\bs{x}\in\bs{\mc X}$.



Since we consider a SGNEP, the feasible decision set of each agent $i \in \mc I$ is denoted by the set-valued mapping $\mc X_i$  as in \eqref{constr} and the collective feasible set is $\bs{\mc X}$ as in \eqref{collective_set}.
We suppose that there is no uncertainty in the constraints and that Assumption \ref{ass_G} and \ref{ass_X} are satisfied.

The aim of each agent $i$, given the decision variables of the other agents $\bs{x}_{-i}$, is to choose a strategy $x_i$, that solves its local optimization problem, i.e.,
\begin{equation}\label{eq_game_stoc}
\forall i \in \mc I: \quad \left\{\begin{array}{cl}
\min\limits _{x_i \in \Omega_i} & \JJ_i\left(x_i, \bs{x}_{-i}\right) \\ 
\text { s.t. } & A_i x_i \leq b-\sum_{j \neq i}^{N} A_j x_j.
\end{array}\right.
\end{equation}
We aim to compute a stochastic Generalized Nash equilibrium (SGNE) as in Definition \ref{def_GNE} but with expected value cost functions, that is, a collective strategy $\bs x^*\in\bs{\mc X}$ such that for all $i \in \mc I$
\begin{equation}\label{eq_SGNE}
\JJ_i(x_i^{*}, \boldsymbol x_{-i}^{*}) \leq \inf \{\JJ_i(y, \boldsymbol x_{-i}^{*})\; | \; y \in \mc{X}_i(\boldsymbol x_{-i}^{*})\}.
\end{equation}

To guarantee the existence of a SGNE, we make further assumptions on the cost function.
\begin{standassumption}[Cost functions measurability]\label{ass_J_exp}
For each $i\in\mc I$ and for each $\xi \in \Xi$, the function $f_{i}(\cdot,\boldsymbol x_{-i},\xi)$ is convex, Lipschitz continuous, and continuously differentiable. The function $f_{i}(x_i,\bs x_{-i},\cdot)$ is measurable and for each $\boldsymbol x_{-i}$, the Lipschitz constant $\ell_i(\boldsymbol x_{-i},\xi)$ is integrable in $\xi$.
\fineass\end{standassumption}

While, under Standing Assumptions \ref{ass_G}--\ref{ass_J_exp}, existence of SGNE of the game is guaranteed by \cite[Section 3.1]{ravat2011}, uniqueness does not hold in general \cite[Section 3.2]{ravat2011}.

As for the deterministic case we seek for a v-GNE, we here study the associated stochastic variational inequality (SVI). The (pseudo) gradient mapping is given, in this case, by
\begin{equation}\label{eq_grad_stoc}
\FF(\bs{x})=\op{col}\left((\EE[\nabla_{x_{i}} f_{i}(x_{i}, \bs{x}_{-i},\xi)])_{i\in\mc I}\right).
\end{equation} 
The possibility to exchange the expected value and the gradient is guaranteed by Standing Assumption \ref{ass_J_exp}. The associated SVI reads as
\begin{equation}\label{eq_svi}
\langle \FF(\bs x^*),\bs x-\bs x^*\rangle+G(\bs x)-G(\bs x^*)\geq 0\text { for all } \bs x \in \bs{\mc X}
\end{equation}
where $G$ is defined as in \eqref{eq_G}.

The stochastic v-GNE (v-SGNE) of the game in (\ref{eq_game_stoc}) is defined as the solution of the $\op{SVI}(\bs{\mc X},\FF,G)$ in (\ref{eq_svi}) where $\FF$ is described in (\ref{eq_grad_stoc}) and $\bs{\mc X}$ is defined in (\ref{collective_set}). 

In what follows, we recast the SGNEP into a monotone inclusion. For each agent $i\in\NN$, the Lagrangian function $\mc L_i$ is defined as
$$\mc L_i\left(\bs{x}, \lambda_i\right) :=\JJ_i\left(x_i, \bs{x}_{-i}\right)+g_i\left(x_i\right)+\lambda_i^{\top}(A \bs{x}-b)$$
where $\lambda_i \in \RR_{ \geq 0}^{m}$ is the Lagrangian dual variable associated with the coupling constraints. 
The set of strategies $\bs x^{*}$ is a SGNE if and only if the following KKT conditions are satisfied:
\begin{equation}\label{game_kkt}
\left\{\begin{array}{l}
0 \in \EE[\nabla_{x_i} f_i(x_i^{*}, \bs{x}_{-i}^{*},\xi)]+\partial g_i\left(x_i^{*}\right)+A_i^{\top} \lambda_i \\ [5pt]
0\in (A\bs{x}^*-b)+\op N_{\RR^m_{\geq 0}}(\lambda^*).
\end{array}\right.
\end{equation}


Similarly, we can use the KKT conditions to characterize the variational problem, studying the Lagrangian function associated to the SVI. Since $\bs x^{*}$ is a solution of $\op{SVI}(\bs{\mc X},\FF,G)$ if and only if 
$$\bs x^{*} \in \stackbin[\bs y \in \bs{\mc{X}}]{}{\op{argmin}}\left(\bs y-\bs x^{*}\right)^{\top} \FF\left(\bs x^{*}\right),$$ 
the associated KKT optimality conditions reads as
\begin{equation}\label{VI_KKT}
 \forall i \in \mc I : \quad\begin{cases}
0 \in \EE[\nabla_{x_i} f_i(x_i^{*}, \bs{x}_{-i}^{*},\xi)]+\partial g_i\left(x_i^{*}\right)+A_i^{\top} \lambda,\\ 
0\in (A\bs{x}^*-b)+\op N_{\RR^m_{\geq 0}}(\lambda^*).
\end{cases}
\end{equation}
As exploited by \cite[Theorem 3.1]{facchinei2007vi}, \cite[Theorem 3.1]{auslender2000}, we seek for a v-SGNE, that is, an equilibrium that reach consensus of the dual variables.

\subsection{Stochastic preconditioned forward-backward algorithm}\label{sec_algo_SGNEP}

We now describe the details of the preconditioning procedure that leads to the distributed iterations presented in Algorithm \ref{algo_FB_stoc_i}.

\begin{algorithm}
\caption{Preconditioned Stochastic Forward--Backward}\label{algo_FB_stoc_i}
Initialization: $x_i^0 \in \Omega_i, \lambda_i^0 \in \RR_{\geq0}^{m},$ and $z_i^0 \in \RR^{m} .$\\
Iteration $k$: Agent $i$\\
(1): Receives $x_j^k$ for all $j \in \mathcal{N}_{i}^{h}, \lambda_{j}^k$ for $j \in \mathcal{N}_{i}^{\lambda}$ then updates:
$$\begin{aligned}
&x_i^{k+1}=\op{prox}_{g_i}[x_i^k-\alpha_{i}(\hat F_{i}(x_i^k, \boldsymbol{x}_{-i}^k,\xi_i^k)-A_{i}^\top \lambda_i^k)]\\
&z_i^{k+1}=z_i^k-\nu_{i} \sum_{j \in \mathcal{N}_{i}^{\lambda}} w_{i,j}(\lambda_i^k-\lambda_{j}^k)\\
\end{aligned}$$
(2): Receives $z_{j, k+1}$ for all $j \in \mathcal{N}_{i}^{\lambda}$ then updates:
$$\begin{aligned}
&\lambda_i^{k+1}=\op{proj}_{\RR_{+}^{m}}\left[\lambda_i^k+\sigma_{i}\left(A_{i}(2x_i^{k+1}-x_i^k)-b_{i}\right)\right.\\
&\qquad+\sigma_{i}\sum_{j \in \mathcal{N}_{i}^{\lambda}} w_{i,j}\left(2(z_i^{k+1}-z_{j}^{k+1})-(z_i^k-z_{j}^k)\right)\\
&\qquad-\sigma_{i}\sum_{j \in \mathcal{N}_{i}^{\lambda}} \left.w_{i,j}(\lambda_i^k-\lambda_{j}^k)\right]\\
\end{aligned}$$
\end{algorithm}

We suppose that each agent $i$ only knows its local data, i.e., $\Omega_i$, $A_i$ and $b_i$. Moreover, each player has access to a pool of samples of the random variable and is able to compute, given the actions of the other players $\bs x_{-i}$, $\EE[\nabla_{x_i} f_i(x_i,\bs{x}_{-i},\xi)]$ (or an approximation, as exploited later in this section). 

Since the cost function is affected by the other agents strategies, we call $\mc N_i^f$ the set of agents $j$ interacting with $i$. Specifically, $j\in\mc N_i^f$ if the function $f_i(x_i,\bs x_{-i})$ explicitly depends on $x_j$.

A local copy of the dual variable is shared through the dual variables graph, $\mc G^\lambda=(\mc I,\mc E^\lambda)$. Along with the dual variable, agents share on $\mc G^\lambda$ a copy of the auxiliary variable $z_i\in\RR^m$. The role of $\bs z=\op{col}(z_1,\dots,z_N)$ is to force consensus, since this is the configuration that we are seeking. A deeper insight on this variable is given later in this section.
The set of edges $\mc E^\lambda$ is given by: $(i,j)\in\mc E^\lambda$ if player $i$ can receive $\{\lambda_j,z_j\}$ from player $j$. The neighbouring agents in $\mc G^\lambda$ forms a set $\mc N^\lambda_i=\{j\in\mc I:(i,j)\in\mc E^\lambda\}$ for all $i\in\mc I$. In this way, each agent control his own decision variable, a local copy of the dual variable $\lambda_i$ and of the auxiliary variable $z_i$ and has access to the other agents variables through the graphs. 

\begin{standassumption}[Graph connectivity]\label{ass_graph}
The dual-variable graph $G_\lambda$ is undirected and connected.
\fineass\end{standassumption}

We call $W\in\RR^{N\times N}$ the weighted adjacency matrix of $\mc G^\lambda$. Then, letting $d_i=\sum_{j=1}^Nw_{i,j}$ and $D=\op{diag}\{d_1,\dots,d_N\}$, the associated Laplacian is the matrix $L=D-W$. Moreover, it follows from Standing Assumption \ref{ass_graph} that $L=L^\top$.
Assumption \ref{ass_graph} is important to reach consensus of the dual variables.

Rewriting the KKT conditions in \eqref{VI_KKT} in compact form as
\begin{equation}\label{eq_T}
0\in\mc T(\bs{x},\bs\lambda):=\left[\begin{array}{c}
G(\bs{x})+\FF(\bs{x})+A^{\top} \lambda \\ 
\mathrm{N}_{\RR_{ \geq 0}^{m}}(\lambda)-(A \bs{x}-b)
\end{array}\right]
\end{equation}
where $\mc T:\mc X\times \RR^m_{\geq 0}\rightrightarrows \RR^{n}\times\RR^m$ is a set-valued mapping, it follows that the v-SGNE correspond to the zeros of the mapping $\mc T$. 

In the remaining part of this section, we split $\mc T$ into the summation of two operators $\mc A$ and $\mc B$ that satisfy specific properties. The advantage of this technique is that the zeros of the mapping $\mc A+\mc B$ correspond to the fixed point of a specific operator depending on both $\mc A$ and $\mc B$, as exploited in \cite{belgioioso2018,yi2019}. Such a scheme is known as forward backward (FB) splitting \cite[Section 26.5]{bau2011}. Indeed, it holds that, for any matrix $\Phi\succ0$, $\bs\omega\in\op{zer}(\mc A+\mc B)$ if and only if, 
$$\bs\omega=(\mathrm{Id}+\Phi^{-1} \mc B)^{-1} \circ(\mathrm{Id}-\Phi^{-1} \mc A)(\bs\omega).$$

Specifically, the operator $\mc T$ can be written as a summation of the two operators
\begin{equation}\label{eq_splitting}
\begin{aligned}
\mc{A} &:\left[\begin{array}{l}
\bs{x} \\
\lambda
\end{array}\right] \mapsto\left[\begin{array}{c}
\FF(\bs{x}) \\ 
b
\end{array}\right]\\
\mc{B} &:\left[\begin{array}{l}
\bs{x} \\ 
\lambda
\end{array}\right] \mapsto\left[\begin{array}{c}
G(\bs{x}) \\ 
\mathrm{N}_{\RR_{ \geq 0}^{m}}(\lambda)
\end{array}\right]+\left[\begin{array}{cc}
0 & A^{\top} \\ 
-A & 0
\end{array}\right]\left[\begin{array}{l}
\bs{x} \\ 
\lambda
\end{array}\right].
\end{aligned}
\end{equation}

Therefore, finding a solution of the variational SGNEP translates in finding a pair $(\bs x^*,\bs \lambda^*)\in\bs{\mc X}\times \RR_{\geq0}^m$ such that $(\bs x^*,\bs\lambda^*)\in\op{zer}(\mc A+\mc B)$.

To impose consensus on the dual variables, the authors in \cite{yi2019} proposed the Laplacian constraint ${\bf{L}}\bs\lambda=0$. This is why, to preserve monotonicity we expand the two operators $\mc A$ and $\mc B$ in \eqref{eq_splitting} introducing the auxiliary variable $\bs z\in\RR^{Nm}$.
Let $L\in\RR^{N\times N}$ be the laplacian of $\mc G^\lambda$  and set ${\bf{L}}=L\otimes \op{Id}_m\in\RR^{Nm\times Nm}$. Let us define ${\bf{A}}=\op{diag}\{A_1,\dots,A_N\}\in\RR^{Nm\times n}$ and $\bs \lambda=\op{col}(\lambda_1,\dots,\lambda_N)\in\RR^{Nm}$; similarly we define $\bs b$ of suitable dimensions. Then, let us define

\begin{equation}\label{eq_expanded_stoc}
\begin{aligned}
\bar{\mc{A}} &:\left[\begin{array}{l}
\bs{x} \\
\bs z\\
\bs \lambda
\end{array}\right] \mapsto\left[\begin{array}{c}
\FF(\bs{x}) \\ 
0\\
\bs b
\end{array}\right]+\left[\begin{array}{c}
0 \\ 
0\\
{\bf{L}}\bs\lambda
\end{array}\right]\\
\bar{\mc{B}} &:\left[\begin{array}{l}
\bs{x} \\ 
\bs z\\
\bs \lambda
\end{array}\right] \mapsto\left[\begin{array}{c}
G(\bs{x}) \\ 
{\bf{0}}\\
\mathrm{N}_{\RR_{ \geq 0}^{m}}(\lambda)
\end{array}\right]+\left[\begin{array}{ccc}
0 & 0 & {\bf{A}}^{\top} \\ 
0 & 0 & {\bf{L}}\\
-{\bf{A}} & -{\bf{L}} & 0
\end{array}\right]\left[\begin{array}{l}
\bs{x} \\ 
\bs z\\
\bs \lambda
\end{array}\right].
\end{aligned}
\end{equation}

To ensure that the zeros of $\bar{\mc A}+\bar{\mc B}$ correspond to the zeros of the operator $\mc T$ in \eqref{eq_T}, we take the following assumption.
\begin{assumption}[Restricted cocoercivity]\label{ass_res_coco}
$\FF$ is restricted $\beta$-cocoercive, with $\beta>0$. 
\fineass\end{assumption}
Then, the following result holds.
\begin{lemma}\label{lemma_zero}
Let Assumption \ref{ass_res_coco} hold and consider the operators $\bar{\mc A}$ and $\bar{\mc B}$ in (\ref{eq_expanded_stoc}), and the operators $\mc A$ and $\mc B$ in (\ref{eq_splitting}). Then the following hold.
\begin{enumerate}
\item[(i)] Given any $\bs\omega^* \in \op{zer}(\bar{\mc A}+\bar{\mc B})$, $\bs{x}^{*}$ is a v-GNE of game in (\ref{game_det}), i.e., $\bs{x}^*$ solves the $\op{VI}(\bs{\mc X},\FF,G)$ in (\ref{eq_vi}). Moreover $\bs{\lambda}^{*}=\mathbf{1}_{N} \otimes \lambda^{*},$ and $(\bs{x}^*,\lambda^{*})$ satisfy the KKT condition in (\ref{VI_KKT}) i.e., $\op{col}(\bs{x}^*, \lambda^{*}) \in \op{zer}(\mc A+\mc B)$;
\item[(ii)] $\op{zer}(\mc A+\mc B) \neq \emptyset$ and $\op{zer}(\bar{\mc A}+\bar{\mc B}) \neq \varnothing$.
\end{enumerate}
\end{lemma}
\begin{proof} See Appendix \ref{proofs_op}.
\end{proof}

The two operators $\bar{\mc A}$ and $\bar{\mc B}$ in \eqref{eq_expanded_stoc} have the following properties.

\begin{lemma}\label{lemma_op}
Let Assumption \ref{ass_res_coco} hold and let $\Phi\succ 0$. The operators $\bar{\mc A}$ and $\bar{\mc B}$ in \eqref{eq_expanded_stoc} have the following properties:
\begin{enumerate}
\item[(i)] $\bar{\mc A}$ is $\theta$-cocoercive where $0<\theta \leq\min \left\{\frac{1}{2 d^{*}}, \beta\right\}$ and $d^*$ is the maximum weighted degree of $\mc G^\lambda$;
\item[(ii)] The operator $\bar{\mc B}$ is maximally monotone;
\item[(iii)] $\Phi^{-1}\bar{\mc A}$ is $\theta\gamma$-cocoercive where $\gamma=\frac{1}{|\Phi^{-1}|}$;
\item[(iv)] $\Phi^{-1}\bar{\mc B}$ is maximally monotone.
\end{enumerate}
\end{lemma}
\begin{proof} See Appendix \ref{proofs_op}.
\end{proof}


Since the expected value can be hard to compute, as the distribution of the random variable is unknown, we take an approximation of the pseudogradient. At this stage, it is not important which type of approximation we use, therefore, in what follows, we replace $\bar{\mc A}$ with 
\begin{equation}\label{eq_A_approx}
\hat{\mc A}:\left[\begin{array}{c}
(\bs{x},\xi) \\
\bs z\\
\bs \lambda
\end{array}\right] \mapsto\left[\begin{array}{c}
\hat F(\bs{x},\xi) \\ 
0\\
\bs b
\end{array}\right]+\left[\begin{array}{c}
0 \\ 
0\\
{\bs{L}}\bs\lambda
\end{array}\right]
\end{equation}
where $\hat F$ is an approximation of the expected value mapping $\FF$ in (\ref{eq_grad_stoc}) given some realization of the random vector $\xi$. 

The fixed point problem, given $\Phi\succ0$, now reads as
\begin{equation}\label{eq_fix_stoc}
\bs\omega=(\op{Id}+\Phi^{-1}\bar{\mc B})^{-1}\circ(\op{Id}-\Phi^{-1}\hat{\mc A})(\bs \omega)
\end{equation}
and suggests the stochastic FB algorithm
\begin{equation}\label{fixed_point_stoc}
\bs\omega^{k+1}=(\mathrm{Id}+\Phi^{-1} \mathcal{B})^{-1} \circ(\mathrm{Id}-\Phi^{-1} \hat{\mc A})(\bs\omega^k,\xi^k).
\end{equation}
where $(\mathrm{Id}+\Phi^{-1} \bar{\mathcal{B}})^{-1}$ represent the backward step and $(\mathrm{Id}-\Phi^{-1} \bar{\mc A})$ is the forward step.

By expanding (\ref{fixed_point_stoc}), we obtain the distributed FB steps in Algorithm \ref{algo_FB_stoc_i} with $\hat{\mc A}$ as in \eqref{eq_A_approx}, $\bar{\mc B}$ as in \eqref{eq_expanded_stoc} and 
\begin{equation}\label{eq_phi_k}
\Phi=\left[\begin{array}{ccc}
\alpha^{-1} & 0 & -{\bf{A}}^\top\\
0 & \nu^{-1} & -{\bf{L}}\\
-{\bf{A}} & -{\bf{L}} & \sigma^{-1}
\end{array}\right].
\end{equation} 
$\alpha^{-1}=\op{diag}\{\alpha_1^{-1}\op{I}_{n_1},\dots,\alpha_N^{-1}\op{I}_{n_N}\}\in\RR^{n\times n}$ and similarly we define $\sigma^{-1}$ and $\nu^{-1}$ of suitable dimensions.
We note that $\Phi$ is symmetric and such that $\bs\omega^k$ is easy to be computed and the iterations are sequential \cite{belgioioso2018}.
If we use the traditional FB algorithm with $\Phi=\op{Id}$, we have to compute the resolvent of $\bar{\mc B}$ that involves the constraint matrix $A$ and the Laplacian $L$, therefore, it could not be evaluated in a distributed way. With $\Phi$ as in \eqref{eq_phi_k}, we overcome this problem \cite{yi2019}.


\section{Convergence analysis with sample average approximation}\label{sec_conv}

We here state some sufficient assumptions for the convergence of Algorithm \ref{algo_FB_stoc_i} to a v-SGNE. Note that Algorithm \ref{algo_FB_stoc_i} involves an approximation of $\FF$ but it does not specify which one. Indeed, preconditioning can be done independently of the approximation scheme. Concerning the convergence analysis, however, we consider the sample average approximation (SAA) scheme. 


We assume the decision maker to have access to an increasing number $S_k$ of samples of the random variable $\xi$ and to be able to compute an approximation of $\FF(x)$ of the form
\begin{equation}\label{eq_F_SAA}
\begin{aligned}
&\hat F(\bs x,\bs\xi)=F_\textup{SAA}(\bs x,\bs\xi)\\
&=\op{col}\left(\frac{1}{S_k} \sum_{t=1}^{S_k} \nabla_{x_1}f_1(\bs x,\xi_1^{(t)}),\dots,\frac{1}{S} \sum_{t=1}^{S_k} \nabla_{x_N}f_N(\bs x, \xi_N^{(t)})\right).
\end{aligned}
\end{equation}
where $\bs \xi=\op{col}(\bar \xi_1,\dots,\bar\xi_n)$, for all $i\in\mc I$, $\bar\xi_i=\op{col}(\xi_i^{(1)},\dots,\xi_i^{(S_k)})$ and $\bs \xi$ is an i.i.d. sequence of random variables drawn from $\PP$.

Approximations of the form (\ref{eq_F_SAA}) are very common in Monte-Carlo simulation approaches, machine learning and computational statistics \cite{staudigl2019}. 


For $k \geq 0,$ let us introduce the approximation error 
\begin{equation}\label{eq_stoc_error}
\epsilon_k=F_\textup{SAA}(\bs x^k,\xi_k)-\FF(\bs x^k).
\end{equation}
\begin{remark}\label{remark_error}
Since there is no uncertainty in the constraints
$$\mc A_\textup{SAA}(\bs\omega^k,\xi^k)-\mc A(\bs\omega^k)=\varepsilon_k,$$
where $\mc A_\textup{SAA}$ is the operator $\hat{\mc A}$ with approximation $F_\textup{SAA}$ as in \eqref{eq_F_SAA} and $\varepsilon_k=\op{col}(\epsilon_k,0,0)$.
\fineass\end{remark}
The following assumption is widely used in the stochastic framework \cite{koshal2013,iusem2017}.
\begin{standassumption}[Zero mean error]\label{ass_error}
For al $k\geq 0$, 
$$\EEk{\epsilon_k}=0 \quad \op{ a.s.}\vspace{-.65cm}$$
\fineass\end{standassumption}

To guarantee that $\Phi$ is positive definite and to obtain convergence, the step size sequence can be taken constant but it should satisfy some bounds \cite[Lemma 6]{yi2019}.
\begin{assumption}[Bounded step sizes]\label{ass_bound}
The step size sequence is such that, given $\gamma>0$, for every agent $i\in\mc I$
$$\begin{aligned}
0&<\alpha_{i} \leq\left(\gamma+\max _{j\in\{1, \ldots, n_{i}\}}\sum\nolimits_{k=1}^{m}|[A_{i}^\top]_{j k}|\right)^{-1} \\ 
0&<\nu_{i} \leq\left(\gamma+2 d_{i}\right)^{-1}\\
0&<\sigma_{i} \leq\left(\gamma+2 d_{i}+\max _{j\in\{1, \ldots, m\}}\sum\nolimits_{k=1}^{n_{i}}|[A_{i}]_{j k}|\right)^{-1}\\
\end{aligned}$$
where $[A_i^\top]_{jk}$ indicates the entry $(j,k)$ of the matrix $A_i^\top$.
Moreover,
$$\norm{\Phi^{-1}}<2\theta,$$
where $\theta$ is the cocoercivity constant of $\bar{\mc A}$ as in Lemma \ref{lemma_op}.
\fineass\end{assumption}

The number of samples to be taken for the SAA must satisfy some conditions as well.
\begin{assumption}[Increasing batch size]\label{ass_batch}
The batch size sequence $(S_k)_{k\geq 1}$ is such that, for some $c,k_0,a>0$,
$$S_k\geq c(k+k_0)^{a+1}.\vspace{-.65cm}$$
\fineass\end{assumption}
This assumption implies that $1/S_k$ is summable, which is a standard assumption in SAA schemes. It is often used in combination with the forthcoming variance reduction assumption to control the stochastic error \cite{staudigl2019,iusem2017}.

\begin{assumption}[Variance reduction]\label{ass_variance}
There exist $p \geq 2$, $ \sigma_{1} \geq 0,$ and a measurable locally bounded function $\sigma : \op{SOL}(\bs{\mc X},\FF,G) \rightarrow \mathbb{R} $ such that for all $(\bs x,\bs x^*) \in \mathbb{R}^{n}\times\op{SOL}(\bs{\mc X},\FF,G)$ 
\begin{equation}\label{eq_var_red}
\EE\left[\norm{F_\textup{SAA}(\bs x, \cdot)-\FF(\bs x)}^{p}\right]^{\frac{1}{p}} \leq \sigma\left(\bs x^{*}\right)+\sigma_{1}\left\|\bs x-\bs x^{*}\right\|.
\end{equation}
\fineass\end{assumption}

\begin{remark}
For simplicity of presentation, let us consider a stronger assumption instead of \eqref{eq_var_red}, namely, for all $\bs x\in\bs{\mc X}$ 
\begin{equation}\label{variance}
\EEx{\norm{F_\textup{SAA}(\bs x, \cdot)-\FF(\bs x)}^{2}}\leq \sigma^2
\end{equation}
for some $\sigma>0$. In the literature, \eqref{variance} is known as uniform bounded variance. Assumption \ref{ass_variance} is more natural when the feasible set is unbounded and it is always satisfied when mapping $\FF$ is Caratheodory and Lipschitz continuous \cite[Ex. 3.1]{staudigl2019}. Since we are in a game theoretic setup, our feasible set is bounded, we can to use (\ref{variance}) as a variance control assumption. 

We also remark that all the following results hold also in the more general case given by Assumption \ref{ass_variance} and using the $L_p$ norm for any $p\geq 2$. We refer to \cite{staudigl2019, iusem2017} for a more detailed insight on this general case.
\fineass\end{remark}

We are now ready to state the convergence result for the SAA case.
\begin{theorem}\label{theo_SAA_sgne}
Let Assumptions \ref{ass_G}--\ref{ass_variance} hold. Then the sequence $(x^k)_{k\in\NN}$ generated by Algorithm \ref{algo_FB_stoc_i} with approximation $\hat F=F_{\textup{SAA}}$ as in \eqref{eq_F_SAA} converges a.s. to a $\mathrm{v}$-$\mathrm{SGNE}$ of the game in \eqref{eq_game_stoc}. 
\end{theorem}
\begin{proof}
See Appendix \ref{sec_proofs_SGNEPs}.
\end{proof}

\section{Stochastic Nash equilibrium problem}\label{sec_SNEPs}
\subsection{Stochastic Nash equilibrium recap}
In this section we consider a SNEP, that is, a SGNEP with expected value cost functions but without shared constraints.

We consider a set $\mc I=\{1,\dots,N\}$ of noncooperative agents choosing their strategy $x_i\in\RR^{n_i}$ from its local decision set $\Omega_i$ which satisfy Standing Assumption \ref{ass_G}. The local cost function of agent $i$ is defined as in \eqref{eq_cost_stoc}.

Standing Assumptions \ref{ass_G}, \ref{ass_J_det} and \ref{ass_J_exp} hold also in this case. 
The aim of each agent $i$, given the decision variables of the other agents $\bs{x}_{-i}$, is to choose a strategy $x_i$, that solves its local optimization problem, i.e.,
\begin{equation}\label{eq_game_SNE}
\forall i \in \mc I: \quad \min\limits _{x_i \in \Omega_i}  \JJ_i\left(x_i, \bs{x}_{-i}\right).
\end{equation}
As a solution, we aim to compute a stochastic Nash equilibrium (SNE), that is, a collective strategy $\bs x^*\in\bs{\Omega}$ such that for all $i \in \mc I$
$$\JJ_i(x_i^{*}, \boldsymbol x_{-i}^{*}) \leq \inf \{\JJ_i(y, \boldsymbol x_{-i}^{*})\; | \; y \in \Omega_i)\}.$$
We note that, compared to Definition \ref{def_GNE} and Equation \eqref{eq_SGNE}, here we consider only local constraints.

Also in this case, we study the associated stochastic variational inequality (SVI) given by 
\begin{equation}\label{eq_svi_ne}
\langle \FF(\bs x^*),\bs x-\bs x^*\rangle+G(\bs x)-G(\bs x^*)\geq 0\text { for all } \bs x \in \bs{\Omega}
\end{equation}
where $\FF$ is defined in \eqref{eq_grad_stoc} and $G$ as in \eqref{eq_G}. As mentioned in Remark \ref{remark_NE}, being a NE is necessary and sufficient for being a solution of the $\VI(\bs\Omega,\FF,G)$.

The stochastic variational equilibrium (v-SNE) of game (\ref{eq_game_SNE}) is defined as the solution of the $\op{SVI}(\bs\Omega , \FF,G)$ in (\ref{eq_svi}) where $\FF$ is described in (\ref{eq_grad_stoc}). The distributed FB algorithm that we propose is presented in Algorithm \ref{algo_FB_sne_i}. 

\begin{algorithm}
\caption{Distributed Stochastic Forward--Backward}\label{algo_FB_sne_i}
Initialization: $x_i^0 \in \Omega_i$\\
Iteration $k$: Agent $i$ receives $x_j^k$ for all $j \in \mathcal{N}_{i}^{h}$, then updates:
$$x_i^k=\op{prox}_{g_i}[x_i^k-\alpha_{i}\hat F_{i}(x_i^k, \boldsymbol{x}_{-i}^k,\xi_i^k)]$$
\end{algorithm}


\subsection{Convergence for restricted cocoercive pseudogradients}

If the restricted cocoercivity assumption holds for the pseudogradient (Assumption \ref{ass_res_coco}) and there are enough sample available, one can use Algorithm \ref{algo_FB_stoc_i} with $A=0, b=0$ and the SAA scheme as in \eqref{eq_F_SAA}. Moreover in this case, it is possible to use also the stochastic approximation scheme. 

In this case, we approximate $\FF$ with only one realization of the random variable $\xi$, therefore, the approximation $\hat F$ is formally defined as \begin{equation}\label{eq_F_SA}
\begin{aligned}
\hat F(\bs x^k,\bs \xi^k)&=F_\textup{SA}(\bs x^k,\bs \xi^k)\\
&=\op{col}\left(\nabla_{x_1}f_1(\bs x^k,\xi^k_1),\dots,\nabla_{x_N}f_N(\bs x^k,\xi^k_N)\right),
\end{aligned}
\end{equation}
where $\bs\xi^{k} =\op{col}(\xi_1^{k},\dots,\xi_{N}^k)\in\RR^N$ is a collection of i.i.d. random variables drawn from $\PP$.

Before stating the convergence result, we state some further assumptions.
With a little abuse of notation, the approximation error is defined as
$$\epsilon_k=F_\textup{SA}(\bs x^k,\xi_k)-\FF(\bs x^k).$$

We suppose that it satisfies Standing Assumption \ref{ass_error}. Moreover, in the SA scheme, there are assumptions also on the step size sequence. In particular, here we let the sequence of the step sizes to be diminishing. This assumption is standard in literature and it has the role of controlling the stochastic error  \cite{koshal2013,koshal2010}. 
\begin{assumption}[Vanishing step sizes]\label{ass_step}
The step size sequence $(\gamma_k)_{k\in\NN}$ is such that
$$\sum_{k=0}^\infty\gamma_k=\infty, \;\sum_{k=0}^\infty\gamma_k^2<\infty \text{ and }\sum_{k=0}^\infty\gamma_k^2\,\EEk{\normsq{\epsilon_k}}<\infty.\vspace{-.2cm}$$
\fineass\end{assumption}

\begin{assumption}[Bounded step sizes]\label{ass_bound_step}
The step size sequence $\{\gamma_k\}$ is such that $\gamma_k\leq2\beta$ where $\beta$ is the cocoercivity constant of $\FF$ as in Assumption \ref{ass_res_coco}.
\fineass\end{assumption}
We can now state our convergence result.

\begin{theorem}\label{theo_SA_sne_coco}
Let Assumptions \ref{ass_res_coco}, \ref{ass_step}, \ref{ass_bound_step} hold. The sequence $(x_k)_{k\in\NN}$ generated by Algorithm \ref{algo_FB_sne_i}, with approximation $\hat F=F_{\textup{SA}}$ as in \eqref{eq_F_SA}, converges a.s. to a v-SNE of the game in \eqref{eq_game_SNE}.
\end{theorem}
\begin{proof}
See Appendix \ref{sec_proofs_SNE}.
\end{proof}

\subsection{Convergence for restricted strictly monotone mappings}

If Assumption \ref{ass_res_coco} is replaced with restricted strict monotonicity, it is still possible to prove convergence. Therefore, in the remaining part of this section we analyze this case.


\begin{assumption}\label{ass_res_strict}
$\FF$ is restricted strictly monotone at $\bs x^*\in\op{SOL}(\bs\Omega,\FF,G)$.
\fineass\end{assumption}

Under this assumption, if a solution exists \cite[Corollary 2.2.5]{facchinei2007}, then it is unique \cite[Theorem 2.3.3]{facchinei2007}.

\begin{remark}
A cocoercive mapping is not necessarily strictly monotone and, vice versa, a strictly monotone mapping is not necessarily cocoercive.
\fineass\end{remark}

For the sake of proving convergence, we suppose $\FF$ to be restricted Lipschitz continuous. On the other hand, we recall that it is not necessary to know the value of the Lipschitz constant since it does not affect the parameters involved in the algorithm. This is very practical, since in general the Lipschitz constant is not easy to compute. 
\begin{assumption}\label{ass_res_lip}
$\FF$ is restricted Lipschitz continuous at $\bs x^*\in\op{SOL}(\bs\Omega,\FF,G)$.
\fineass\end{assumption}

We can now state the convergence result.
\begin{theorem}\label{theo_SA_sne_strict}
Let Assumptions \ref{ass_step}, \ref{ass_res_strict}, \ref{ass_res_lip} hold. The sequence $(x_k)_{k\in\NN}$ generated by Algorithm \ref{algo_FB_sne_i} with approximation $\hat F=F_{\textup{SA}}$ as in \eqref{eq_F_SA} converges a.s. to a $\mathrm{v}$-$\mathrm{SNE}$ of the game in \eqref{eq_game_SNE}.
\end{theorem}
\begin{proof}
It follows as the proof of \cite[Proposition 1]{koshal2010}, by substituting the projection with the proximal operator.
\end{proof}

\begin{corollary}
If Assumptions \ref{ass_step}, \ref{ass_res_strict}, \ref{ass_res_lip} hold. The sequence generated by Algorithm \ref{algo_FB_sne_i} with approximation $\hat F=F_{\textup{SAA}}$ as in \eqref{eq_F_SAA} converges a.s. to a v-SNE of the game in \eqref{eq_game_SNE}.\fineass
\end{corollary}
\begin{remark}
Since we take a vanishing step size (Assumption \ref{ass_step}), in the case of Corollary 1, the batch size sequence should not be increasing, that is, we can take a constant number of realizations at every iteration.
\fineass\end{remark}

\section{An open problem}\label{sec_disc}

We now discuss the SGNEP using the stochastic approximation (SA) scheme. The idea is to approximate $\FF$ with only one realization of the random variable $\xi$ as in \eqref{eq_F_SA}. To counterbalance that, it turns that we have to make some further assumptions on $\FF$ and to take a vanishing step size in order to obtain convergence.



In the case with SAA scheme (Section \ref{sec_conv}), one can rely on the huge number of samples to obtain a good approximation of the expected value cost function. When using SA instead, another way to control the error should be found. This is why, usually, we choose a vanishing step size sequence, as in Assumption \ref{ass_step} \cite{koshal2013,jiang2008,kannan2014}. 

Having time varying step size sequence formally translates into having, as a preconditioning matrix,
\begin{equation}\label{eq_phi_k}
\Phi_k=\left[\begin{array}{ccc}
\alpha_k^{-1} & 0 & -{\bf{A}}^\top\\
0 & \nu^{-1} & -{\bf{L}}\\
-{\bf{A}} & -{\bf{L}} & \sigma^{-1}
\end{array}\right],
\end{equation} 
that changes at every iteration.
Therefore, a variable metric, induced by $\Phi_k$ should be used for the convergence analysis. Such analysis is possible considering the results in \cite{yi2019} and \cite{combettes2014, cui2019}. In particular, the FB algorithm in this case reads as
\begin{equation}\label{fixed_point_k}
\bs\omega^{k+1}=(\mathrm{Id}+\Phi_k^{-1} \bar{\mathcal{B}})^{-1} \circ(\mathrm{Id}-\Phi_k^{-1} \hat{\mc A})(\bs\omega^k),
\end{equation}
while the sequence $\{\Phi_k\}$ should satisfy
\begin{equation}\label{eq_metric}
\sup _{k \in \mathbb{N}}\|\Phi_k^{-1}\|<\infty \text { and } \forall k >0 \,\,\,(1+\eta_{k}) \Phi_{k+1}^{-1} \succcurlyeq \Phi_k^{-1}.
\end{equation}
Convergence of \eqref{fixed_point_k} in the deterministic case can be retrieved from \cite[Theorem 4.1]{combettes2014}.

Even if $\alpha_k$ is the only step which should be diminishing, (as in Remark \ref{remark_error}, we have the error only for the primal variable), the variable metric induced with vanishing steps does not satisfy condition \eqref{eq_metric}. Indeed, if Assumption \ref{ass_step} holds, $\norm{\Phi^{-1}_{k+1}}\leq\norm{\Phi^{-1}_{k}}$ and this contradicts condition \eqref{eq_metric}. A possible interpretation of why this condition is not satisfied is, loosely speaking, the following. Even if numerically it is reasonable to have a vanishing step, theoretically, it is not possible to prove if the algorithm converges to a zero of the mapping or to a zero of the steps. However, our numerical experience suggests that the distributed stochastic preconditioned FB algorithm may converge to a SGNE also with the SA scheme (although very slowly).
In light of this considerations, we state the following conjecture.
\begin{conjecture}\label{theo_SA_sgne}
Let Assumptions \ref{ass_res_coco}, \ref{ass_bound}, \ref{ass_step} hold. Then, the sequence generated by Algorithm \ref{algo_FB_stoc_i} with the approximation given in \eqref{eq_F_SA} converges to a v-SGNE of the game in \eqref{eq_game_stoc}.\fineass
\end{conjecture}


In \cite{combettes2014}, the authors propose a deterministic FB algorithm with error and they make the assumption of having a summable sequence of noises. This assumption is not reasonable in the SA approach as in general the approximation error does not vanish with time (this indeed holds in the SAA case, where the second moment of stochastic error vanishes with $1/S_k$, see Proposition \ref{lemma_variance}). On the other hand,  if one has a vanishing stochastic error, the convergence analysis can be done similarly to the SAA case.

\section{Comparative numerical simulations}\label{sec_sim}
Let us propose a number of numerical simulations to corroborate the theoretical analysis where we also compare our algorithm with the forward-backward-forward (FBF) and the extragradient (EG) algorithms.
 
\subsection{Distributed forward-backward-forward and extragradient}
In this section, we describe the distributed FBF scheme, presented in Algorithm \ref{algo_FBF_i} \cite{franci2019fbf} and the distributed EG scheme, in Algorithm \ref{algo_EG_i}.

\begin{algorithm}
\caption{Stochastic Distributed Forward Backward Forward}\label{algo_FBF_i}
Initialization: $x_i^0 \in \Omega_i, \lambda_i^0 \in \RR_{\geq0}^{m},$ and $z_i^0 \in \RR^{m} .$\\
Iteration $k$: Agent $i$\\
($1$) Receives $x_j^k$ for $j \in \mathcal{N}_{i}^{J}$, $ \lambda_j^k$ and $z_{j,k}$ for $j \in \mathcal{N}_{i}^{\lambda}$ then updates
$$\begin{aligned}
&\tilde x_i^k=\op{prox}_{g_i}[x_i^k-\alpha_i(\hat F_{i}(x_i^k,\boldsymbol x_{-i}^k,\xi_i^k)+A_{i}^{\top} \lambda_i^k)]\\
&\tilde z_i^k=z_i^k-\nu_i \sum_{j \in \mathcal{N}_{i}^{\lambda}} w_{i,j}(\lambda_i^k-\lambda_j^k)\\
&\tilde\lambda_i^k=\op{proj}_{\RR^m_{\geq 0}}\{\lambda_i^k+\sigma_i(A_{i}x_i^k-b_{i})\\
&\quad\quad+\sigma_i\sum_{j \in \mathcal{N}_{i}^{\lambda}} w_{i,j}[(z_{i}^{k}-z_j^k)-(\lambda_i^k-\lambda_j^k)]\}
\end{aligned}$$
($2$) Receives $\tilde x_j^k$ for $j \in \mathcal{N}_{i}^{J}$, $ \tilde \lambda_j^k$and $\tilde z_{j,k}$ for $j \in \mathcal{N}_{i}^{\lambda}$ then updates
$$\begin{aligned}
&x_i^{k+1}=\tilde x_i^k-\alpha_i(\hat F_{i}(x_i^k,\boldsymbol x_{-i}^k,\xi_i^k)-\hat F_{i}(\tilde x_i^k,\tilde{\boldsymbol x}_{-i}^k,\eta_i^k))\\
&\quad\quad\quad-\rho_iA_{i}^{\top} (\lambda_i^k-\tilde \lambda_{i,k})\\
&z_i^{k+1}=\tilde z_i^k+\nu_i \sum_{j \in \mathcal{N}_{i}^{\lambda}} w_{i,j}[(\lambda_i^k-\lambda_j^k)-(\tilde\lambda_i^k-\tilde\lambda_j^k)]\\
&\lambda_i^{k+1}=\tilde{\lambda}_i^{k}+\sigma_iA_i(\tilde{x}_{i}^{k}-x_{i}^{k})\\
&\quad\quad\quad-\sigma_i\sum_{j \in \mathcal{N}_{i}^{\lambda}} w_{i,j}[(z_i^k-z_j^k)-(\tilde z_i^k-\tilde z_j^k)]\\
&\quad\quad\quad+\sigma_i\sum_{j \in \mathcal{N}_{i}^{\lambda}} w_{i,j}[(\lambda_{i,k}-\lambda_j^k)-(\tilde\lambda_i^k-\tilde\lambda_j^k)]\\
\end{aligned}$$
\end{algorithm}

To obtain the FBF and EG algorithms, let us rewrite the operator $\bar{\mc B}$ in \eqref{eq_expanded_stoc} as $\bar{\mc B}=\mc C+\mc D$, where $\mc C$ contains the local constraints and $\mc D$ is the skew symmetric matrix. Let us also call $\mc H=\bar{\mc A}+\mc D$.
Then, in compact form, the FBF algorithm generates two sequences $(\bs u^{k},\bs v^{k})_{k\geq 0}$ as follows: 
\begin{equation}\label{FBF}
\begin{aligned}
\bs u^{k}&=\op{J}_{\Psi^{-1} \mc C}(\bs v^{k}-\Psi^{-1} \mc H \bs v^{k})\\
\bs v^{k+1}&=\bs u^{k}+\Psi^{-1} (\mc H\bs v^{k}-\mc H\bs u^{k}).
\end{aligned}
\end{equation}

In \eqref{FBF}, $\Psi$ is a block-diagonal matrix that contains the step sizes: 
\begin{equation}\label{Psi}
\Psi=\op{diag}(\alpha^{-1},\nu^{-1}, \sigma^{-1}),
\end{equation}
where $\alpha$, $\nu$ and $\sigma$ are diagonal matrices of suitable dimensions.
Convergence of the stochastic FBF with the SAA scheme is guaranteed by \cite[Theorem 4.5]{staudigl2019}.

Analogously, we can write the stochastic distributed EG algorithm in compact form as: 
\begin{equation}\label{EG}
\begin{aligned}
\bs u^{k}&=\op{J}_{\Psi^{-1} \mc C}(\bs v^{k}-\Psi^{-1} \mc H \bs v^{k})\\
\bs v^{k+1}&=\op{J}_{\Psi^{-1} \mc C}(\bs v^{k}-\Psi^{-1} \mc H \bs u^{k}).
\end{aligned}
\end{equation}
In this case, convergence is guaranteed by \cite[Theorem 3.18]{iusem2017} with the SAA scheme.
 
\begin{algorithm}
\caption{Distributed Extragradient}\label{algo_EG_i}
Initialization: $x_i^0 \in \Omega_i, \lambda_i^0 \in \RR_{\geq0}^{m},$ and $z_i^0 \in \RR^{m} .$\\
Iteration $k$: Agent $i$\\
($1$) Receives $x_j^k$ for $j \in \mathcal{N}_{i}^{J}$, $ \lambda_j^k$ and $z_{j,k}$ for $j \in \mathcal{N}_{i}^{\lambda}$ then updates
$$\begin{aligned}
&\tilde x_i^k=\op{prox}_{g_i}[x_i^k-\alpha_i(\hat F_{i}(x_i^k,\boldsymbol x_{-i}^k,\xi_i^k)+A_{i}^{\top} \lambda_i^k)]\\
&\tilde z_i^k=z_i^k-\nu_i \sum_{j \in \mathcal{N}_{i}^{\lambda}} w_{i,j}(\lambda_i^k-\lambda_j^k)\\
&\tilde\lambda_i^k=\op{proj}_{\RR^m_{\geq 0}}\{\lambda_i^k+\sigma_i(A_{i}x_i^k-b_{i})\\
&\quad\quad+\sigma_i\sum_{j \in \mathcal{N}_{i}^{\lambda}} w_{i,j}[(z_{i}^{k}-z_j^k)-(\lambda_i^k-\lambda_j^k)]\}
\end{aligned}$$
($2$) Receives $\tilde x_j^k$ for $j \in \mathcal{N}_{i}^{J}$, $ \tilde \lambda_j^k$and $\tilde z_{j,k}$ for $j \in \mathcal{N}_{i}^{\lambda}$ then updates
$$\begin{aligned}
&x_i^{k+1}=\op{prox}_{g_i}[x_i^k-\alpha_i(\hat F_{i}(\tilde x_i^k,\tilde{\boldsymbol x}_{-i}^k,\xi^k_i)+A_{i}^{\top} \tilde\lambda_i^k)]\\
&z_i^{k+1}=z_i^k-\nu_i \sum_{j \in \mathcal{N}_{i}^{\tilde\lambda}} w_{i,j}(\tilde \lambda_i^k-\tilde \lambda_j^k)\\
&\lambda_i^{k+1}=\op{proj}_{\RR^m_{\geq 0}}\{\lambda_i^k+\sigma_i(A_{i}\tilde x_i^k-b_{i})\\
&\quad\quad+\sigma_i\sum_{j \in \mathcal{N}_{i}^{\tilde \lambda}} w_{i,j}[(\tilde z_{i}^{k}-\tilde z_j^k)-(\tilde \lambda_i^k-\tilde \lambda_j^k)]\}
\end{aligned}$$
\end{algorithm}

\subsection{Case study: Network Cournot game}
Now, we consider the electricity market problem as proposed in \cite{xu2013}. Such problem can also be casted as a network Cournot game with markets capacity constraints \cite{yi2019,yu2017}.

We consider a set of $N$ generators (companies) that sell energy in a set of $m$ locations (markets). The random variable $\xi$ represent the demand uncertainty. Each generator decides the quantity $x_i$ of energy to deliver in the $n_i$ markets it is connected with. Each company has a local cost function $c_i(x_i)$ related to the production of electricity. The cost function is not uncertain as we suppose that the companies are able to compute their own cost of production. \\
Each market has a bounded capacity $b_j$ therefore the collective constraints are given by $A\boldsymbol x\leq b$ where $A=[A_1,\dots,A_N]$. Each $A_i$ specifies in which market a company $i$ participates. The prices of the locations are collected in $P:\RR^m\times \Xi\to\RR^m$. The uncertainty variable appears in this functional which is related to the demand. $P$ is supposed to be a linear function. The cost function of each agent is then given by
$$\JJ_i(x_i,x_{-i},\xi)=c_i(x_i)-\EE[P(\xi)^\top(A\boldsymbol x)A_ix_i].$$
Clearly, if $c_i(x_i)$ is strongly convex with Lipschitz continuous gradient and the prices are linear, the pseudo gradient of $\JJ_i$ is strongly monotone.
\subsection{Simulations}\label{sec_sim}

As a numerical setting, we consider a set of 20 companies and 7 markets, similarly to \cite{yi2019,yu2017}. Each company $i$ has a local constraint of the form $0 < x_i < \gamma_i$ where each component of $\gamma_i$ is randomly drawn from $[1, 1.5]$. In terms of electricity market, this can be seen as the capacity limit of generator $i$. Each market $j$ has a maximal capacity $b_j$ randomly drawn from $[0.5, 1]$. The local cost function of the generator $i$ is $c_i(x_i) = \pi_i\sum_{j=1}^{n_i} ([x_i]_j)^2 + q_i^\top x_i$, where $[x_i]_j$ indicates the $j$ component of $x_i$.
$\pi_i$ is randomly drawn from $[1, 8]$, and each component of $q_i$ is randomly drawn from $[0.1, 0.6]$. Notice that $c_i(x_i)$ is strongly convex with Lipschitz continuous gradient. Recall that the cost function of agent $i$ is influenced by the variables of the companies selling in the same market. This information can be retrieved from the graph in \cite[Fig. 1]{yi2019}.

The price $P(\xi) = \bar P-D(\xi)A\boldsymbol x$ is taken as a linear function and each component of $\bar P =\op{col}(\bar P_1,\dots,\bar P_7)$ is randomly drawn from $[2,4]$. The uncertainty appears in the quantities $D(\xi)=\op{diag}\{d_1(\xi_1),\dots, d_7(\xi_7)\}$ that concern the total supply for each market. The entries of $D(\xi)$ are taken with a normal distribution with mean $0.8$ and finite variance. Following \cite{yi2019}, we suppose that the dual variables graph is a cycle graph with the addiction of the edges $(2,15)$ and $(6,13)$. 

We simulate Algorithm \ref{algo_FB_stoc_i}, \ref{algo_FBF_i} and \ref{algo_EG_i} to make a comparison using the SAA scheme. The parameters $\alpha$, $\nu$ and $\sigma$ are taken to be the highest possible to obtain convergence. The plots in Fig. \ref{plot_sol} and \ref{plot_lambda} show respectively $\frac{\norm{\boldsymbol x_{k+1}-\boldsymbol x^{*}}}{\norm{\boldsymbol x^{*}}}$ and $\norm{L \otimes \op{I}_{7} \boldsymbol \lambda_{k}}$. The first performance index indicates the convergence to a solution $\bs x^*$, while the second shows convergence of the dual variables to consensus. The plot in Fig. \ref{plot_time} shows the computational time needed to reach a solution. As one can see from Fig. \ref{plot_sol} and \ref{plot_lambda}, our algorithm is comparable, in terms of number of iterations, to the FBF algorithm but it is computationally the least expensive (Fig. \ref{plot_time}).

Concerning the SA scheme, our numerical experience suggests that a solution may be reached but very slowly. Indeed, taking the same number of iterations as in Fig. \ref{plot_sol}, we obtain a relative distance from the solution of order $10^{-1}$ (higher that the SAA scheme). The advantage of the SA scheme is though that it is computationally not expensive.
 
\begin{figure}[h]
\centering
\includegraphics[width=.45\textwidth]{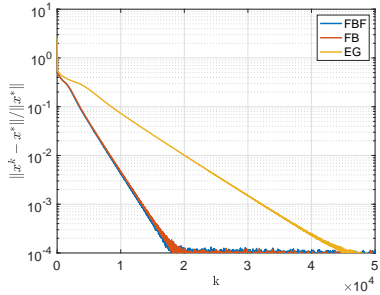}
\caption{Relative distance of the primal variable from the solution.}\label{plot_sol}
\end{figure}

\begin{figure}[h]
\centering
\includegraphics[width=.45\textwidth]{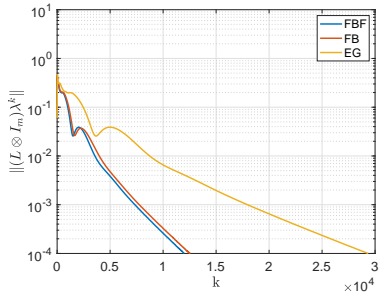}
\caption{Disagreement between the dual variables.}\label{plot_lambda}
\end{figure}

\begin{figure}[h]
\centering
\includegraphics[width=.45\textwidth]{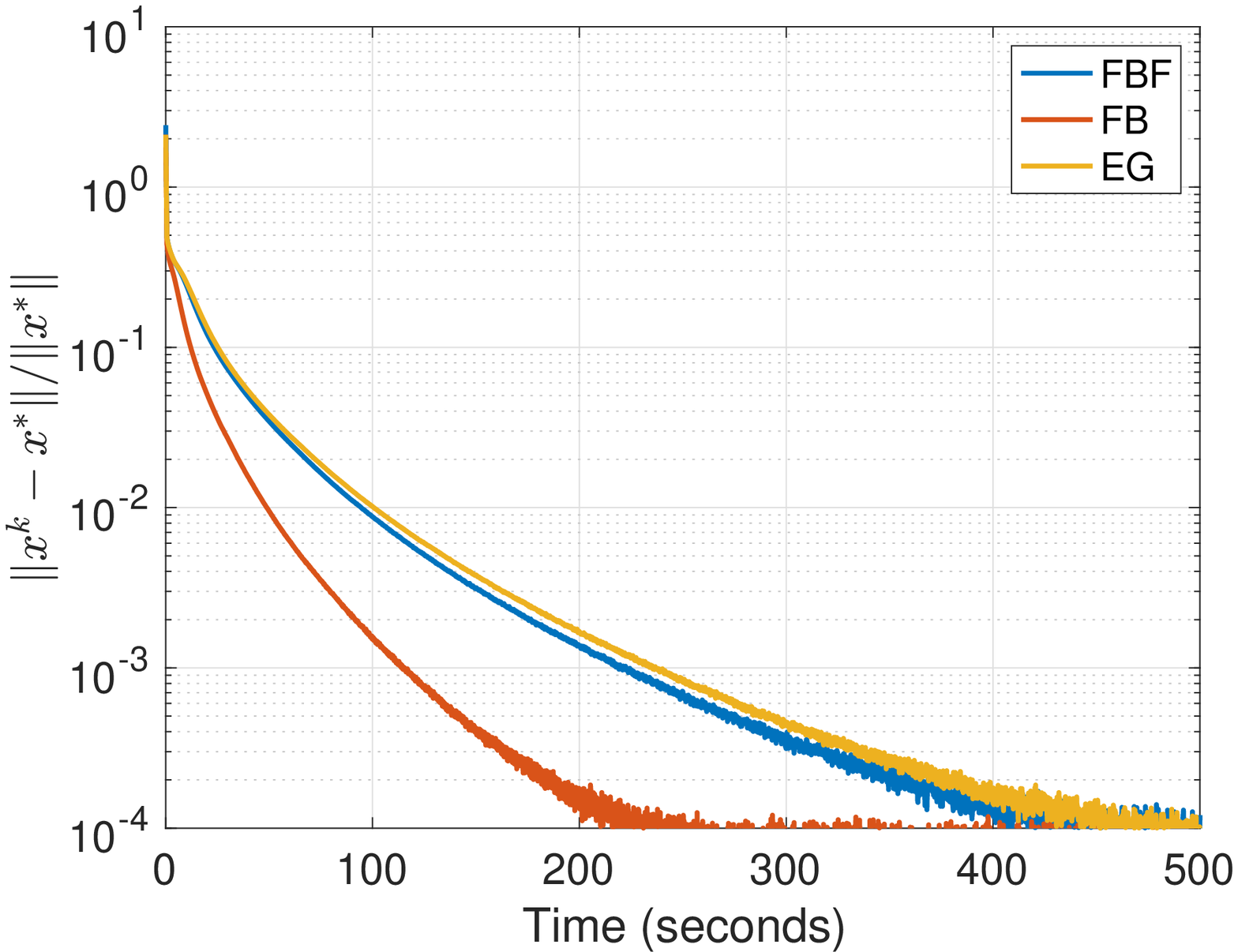}
\caption{Computation time (performed on Matlab R2019a with a 2,3 GHz Intel Core i5 and 8 GB LPDDR3 RAM).}\label{plot_time}
\end{figure}
 
\section{Conclusion}\label{sec_conclu}
The preconditioned forward--backward operator splitting is applicable to stochastic (generalized) Nash equilibrium problems to design distributed equilibrium seeking algorithms. Since the exact expected value is hard to compute in general, stochastic approximation or sample average approximation techniques can be used to ensure convergence almost surely.



Convergence holds under restricted cocoercivity of the pseudo gradient for stochastic generalized Nash equilibrium problems and under restricted strict monotonicity for stochastic Nash equilibrium problems. These assumption are among the weakest known in literature and, remarkably, match well with the deterministic setup.

\section{Appendix A\\ Proofs of Section \ref{sec_algo_SGNEP}}\label{proofs_op}

\begin{proof}[Proof of Lemma \ref{lemma_zero}]
The proof of (i) follows similarly to \cite[Theorem 2]{yi2019}. \\
Concerning (ii), given Standing Assumpitons \ref{ass_G}--\ref{ass_X}, the game in \eqref{eq_game_stoc} has at least one solution $\bs x^*$, therefore, there exists a $\bs\lambda^*\in\RR^m_{\geq0}$ such that the KKT conditions in \eqref{VI_KKT} are satisfied \cite[Theorem 3.1]{auslender2000}. It follows that $\op{zer}(\mc A+\mc B)\neq\emptyset$. The existence of $\bs z^*$ such that $\op{col}(\bs x^*,\bs z^*,\bs\lambda^*)\in\op{zer}(\bar{\mc A}+\bar{\mc B})$ follows from properties of the normal cone and of the Laplacian matrix as a consequence of Assumption \ref{ass_graph} \cite[Theorem 2]{yi2019}.
\end{proof}

\begin{proof}[Proof of Lemma \ref{lemma_op}]
First we notice that $\norm{L}\geq 2d^*$ and that by the Baillon-Haddard Theorem the Laplacian operator is $\frac{1}{2d^*}$-cocoercive. Then Statement (i) follows by this computation:
$$\begin{aligned}
\langle\mc A&(\bs\omega_1)-\mc A(\bs\omega_2),\bs\omega_1-\bs\omega_2\rangle\\
&=\langle \FF(\bs x_1)-\FF(\bs x_2),x_1-x_2\rangle+\langle \bs L\bs \lambda_1-\bs L\bs \lambda_2,\bs \lambda_1-\bs \lambda_2\rangle\\
&\geq\beta\normsq{\FF(\bs x_1)-\FF(\bs x_2)}+\frac{1}{2d^*}\normsq{\bs L\bs \lambda_1-\bs L\bs \lambda_2}\\
&\geq\min\left\{\beta,\frac{1}{2d^*}\right\}\left(\normsq{\FF(\bs x_1)-\FF(\bs x_2)}+\normsq{\bs L\bs \lambda_1-\bs L\bs \lambda_2}\right)\\
&\geq\theta\normsq{\mc A(\bs\omega_1)-\mc A(\bs\omega_2)}.
\end{aligned}$$
The operator $\bar{\mc B}$ is given by a sum, therefore it is maximally monotone if both the addend are \cite[Proposition 20.23]{bau2011}. The first part is maximally monotone by Standing Assumption \ref{ass_G} and Moreau Theorem \cite[Theorem 20.25]{bau2011} and the second part is a skew symmetric matrix \cite[Corollary 20.28]{bau2011}.
Statement (iii) follows from Statement (i) and (iv) follows from (ii) \cite[Lemma 7]{yi2019}.
\end{proof}

\section{Appendix B\\Sequence of random variables}
In this appendix we recall some results on sequences of random variables, given the probability space $(\Xi, \mc F, \PP)$.

Let us define the filtration $\mc F=\{\mc F_k\}$, that is, a family of $\sigma$-algebras such that $\mathcal{F}_{0} = \sigma\left(X_{0}\right)$, for all $k \geq 1$, 
$\mathcal{F}_{k} = \sigma\left(X_{0}, \xi_{1}, \xi_{2}, \ldots, \xi_{k}\right)$
and $\mc F_k\subseteq\mc F_{k+1}$ for all $k\geq0$.

The Robbins-Siegmund Lemma is widely used in literature to prove a.s. convergence of sequences of random variables. It first appeared in \cite{RS1971}. 
\begin{lemma}[Robbins-Siegmund Lemma, \cite{RS1971}]\label{lemma_RS}
Let $\mc F=(\mc F_k)_{k\in\NN}$ be a filtration. Let $\{\alpha_k\}_{k\in\NN}$, $\{\theta_k\}_{k\in\NN}$, $\{\eta_k\}_{k\in\NN}$ and $\{\chi_k\}_{k\in\NN}$ be non negative sequences such that $\sum_k\eta_k<\infty$, $\sum_k\chi_k<\infty$ and let
$$\forall k\in\NN, \quad \EE[\alpha_{k+1}|\mc F_k]+\theta_k\leq (1+\chi_k)\alpha_k+\eta_k \quad a.s.$$
Then $\sum_k \theta_k<\infty$ and $\{\alpha_k\}_{k\in\NN}$ converges a.s. to a non negative random variable.
\end{lemma}

We also need this result for $L_p$ norms, known as Burkholder-Davis-Gundy inequality \cite{stroock2010}.
\begin{lemma}[Burkholder-Davis-Gundy inequality]\label{BDG_Lemma}
Let $\{\mc F_k\}$ be a filtration and $\{U_k\}_{k\geq0}$ a vector-valued martingale relative to this filtration. Then, for all $p\in [1,\infty)$, there exists a universal constant $c_p > 0$ such that for every $k\geq1$
$$\EE\left[\left(\sup _{0 \leq i \leq N}\left\|U_{i}\right\|\right)^{p}\right]^{\frac{1}{p}} \leq c_{p} \EE\left[\left(\sum_{i=1}^{N}\left\|U_{i}-U_{i-1}\right\|^{2}\right)^{\frac{p}{2}}\right]^{\frac{1}{p}}.$$
\end{lemma}
We also recall the Minkowski inequality: for given functions $f,g\in L^p(\Xi,\mc F, \PP)$, $\mc G\subseteq \mc F$ and $p\in[0;\infty]$
$$\EE\left[\norm{f+g}^{p} | \mc{G}\right]^{\frac{1}{p}} \leq \EE\left[\norm{f}^{p} | \mc{G}\right]^{\frac{1}{p}}+\EE\left[\norm{g}^{p} | \mc{G}\right]^{\frac{1}{p}}.$$
When combined with the Burkholder-Davis-Gundy inequality, it leads to the fact that for all $p\geq2$, there exists a constant $c_p >0 $ such that, for every $k\geq1$,
$$\EE\left[\left(\sup _{0 \leq i \leq N}\norm{U_{i}}\right)^{p}\right]^{\frac{1}{p}} \leq c_p\sqrt{\sum_{k=1}^{N} \EE\left(\norm{U_{i}-U_{i-1}}^{p}\right)^{\frac{2}{p}}}.$$

\section{Appendix C\\Proofs of Theorem \ref{theo_SAA_sgne}}\label{sec_proofs_SGNEPs}
In this section, we prove convergence of Algorithm \ref{algo_FB_stoc_i}.

First we note that the iterations of Algorithm \ref{algo_FB_stoc_i} are obtained by expanding (\ref{fixed_point_stoc}) and solving for $\bs x_{k}$, $\bs z_{k}$ and $\bs \lambda_{k}$. Therefore, convergence of the sequence $(\boldsymbol x^k,\boldsymbol\lambda^k)$ to a v-GNE of the game in \eqref{eq_game_stoc} follows by the convergence of the FB iterations in \eqref{fixed_point_stoc}. 

We now prove a preliminary result concerning the sequence generated by Algorithm \ref{algo_FB_stoc_i}.

\begin{lemma}\label{prop_sgne}
Let Assumption \ref{ass_bound}, \ref{ass_error} hold. Then, the sequence generated by \eqref{fixed_point_stoc} satisfy, for some $\zeta>1$ 
\begin{equation}\label{eq_prop}
\begin{aligned}
&\EEk{\normsqphi{\omega^{k+1}-\omega^*}}\leq\normsqphi{\omega^k-\omega^*}+\\
&+2\EEk{\normsqphi{\Phi^{-1}\varepsilon_k}}+\frac{1}{2}\left(\frac{1}{\zeta}-1\right)\op{res}_\Phi(\omega^k)^2.
\end{aligned}
\end{equation}
\end{lemma}
\begin{proof}
We start by using nonexpansiviveness of the resolvent:
\begin{equation*}
\begin{aligned}
&\normsqphi{\omega^{k+1}-\omega^*}\leq
\normsqphi{\omega^k-\omega^*}+2\langle \omega^k-\omega^*,\Phi^{-1}\varepsilon_k\rangle_\Phi\\
&-2\langle \omega^k-\omega^*,\Phi^{-1}(\bar{\mc A}(\omega^k)-\bar{\mc A}(\omega^*))\rangle_\Phi+\\
&-\normsqphi{\omega^k-\omega^{k+1}}+2\langle\omega^k-\omega^{k+1},\Phi^{-1}(\hat{\mc A}(\omega^k)-\bar{\mc A}(\omega^*))\rangle_\Phi\\
\end{aligned}
\end{equation*}
By Young's inequality with $\zeta>1$, we have that
\begin{equation}\label{young}
\begin{aligned}
&2\langle\omega^k-\omega^{k+1},\Phi^{-1}(\hat{\mc A}(\omega^k)-\bar{\mc A}(\omega^{k+1}))\rangle_\Phi\leq\frac{1}{\zeta}\normsqphi{\omega^k-\omega^*}+\\
&+\zeta\normsqphi{\Phi^{-1}(\bar{\mc A}(\omega^k)-\bar{\mc A}(\omega^*))}+\zeta\normsqphi{\Phi^{-1}\varepsilon_k}+\\
&+2\zeta\langle \Phi^{-1}\bar{\mc A}(\omega^k)-\Phi^{-1}\bar{\mc A}(\omega^*),\varepsilon_k\rangle_\Phi
\end{aligned}
\end{equation}
Then, by using cocoercivity of $F$ and including \eqref{young}, we obtain: 
\begin{equation}\label{step}
\begin{aligned}
&\normsqphi{\omega^{k+1}-\omega^*}
\leq\normsqphi{\omega^k-\omega^*}+\zeta\normsqphi{\Phi^{-1}\varepsilon^k}+\\
&+\left(\frac{1}{\zeta}-1\right)\normsqphi{\omega^k-\omega^{k+1}}+2\langle \omega^k-\omega^*,\Phi^{-1}\varepsilon_k\rangle_\Phi+\\
&+\left(\frac{\zeta\norm{\Phi^{-1}}}{\theta}-2\right)\langle \omega^k-\omega^*,\Phi^{-1}\bar{\mc A}(\omega^k)-\Phi^{-1}\bar{\mc A}(\omega^*)\rangle_\Phi\\
&+2\zeta\langle \Phi^{-1}\bar{\mc A}(\omega^k)-\Phi^{-1}\bar{\mc A}(\omega^*),\varepsilon_k\rangle_\Phi
\end{aligned}
\end{equation}
Next, given the residual $\op{res}(\omega^k)$, we have that
\begin{equation}
\begin{aligned}
\op{res}_\Phi&(\omega^k)^2=\normsqphi{\omega^k-(\op{Id}+\Phi^{-1}\mc B)^{-1}(\omega^k-\Phi^{-1}\bar{\mc A}(\omega^k))}\\
\leq&2\normsqphi{\omega^k-\omega^{k+1}}+2\normsqphi{(\op{Id}+\Phi^{-1}\mc B)^{-1}(\omega^k-\Phi^{-1}\hat{\mc A}(\omega^k,\xi^k))+\right.\\
&\left.-(\op{Id}+\Phi^{-1}\mc B)^{-1}(\omega^k-\Phi^{-1}\bar{\mc A}(\omega^k))}\\
\leq&
2\normsqphi{\omega^k-\omega^{k+1}}+2\normsqphi{\Phi^{-1}\varepsilon^k}
\end{aligned}
\end{equation}
where the first equality follows by the definition of $\omega^{k+1}$ and the last inequality follows from non expansivity. Then,
$$\normsqphi{\omega^k-\omega^{k+1}}\geq\frac{1}{2}\op{res}_\Phi(\omega^k)^2-\normsqphi{\Phi^{-1}\varepsilon^k}$$
Finally, equation \eqref{step} becomes
\begin{equation*}
\begin{aligned}
&\normsqphi{\omega^{k+1}-\omega^*}\leq\normsqphi{\omega^k-\omega^*}+\left(\zeta-\frac{1}{\zeta}+1\right)\normsqphi{\Phi^{-1}\varepsilon^k}+\\
&+2\langle \omega^k-\omega^*,\Phi^{-1}\varepsilon_k\rangle_\Phi+\frac{1}{2}\left(\frac{1}{\zeta}-1\right)\op{res}_\Phi(\omega^k)^2\\
&+\left(\frac{\zeta\norm{\Phi^{-1}}}{\theta}-2\right)\langle \omega^k-\omega^*,\Phi^{-1}\bar{\mc A}(\omega^k)-\Phi^{-1}\bar{\mc A}(\omega^*)\rangle_\Phi\\
&+2\zeta\langle \Phi^{-1}\bar{\mc A}(\omega^k)-\Phi^{-1}\bar{\mc A}(\omega^*),\varepsilon_k\rangle_\Phi
\end{aligned}
\end{equation*}
By taking the expected value and by using Assumption \ref{ass_error} and \ref{ass_bound}, we obtain \eqref{eq_prop}.
\end{proof}


Before proving convergence of the algorithm with the SAA scheme, we prove a preliminary result on the variance of the stochastic error.

\begin{lemma}\label{lemma_variance}
For all $k\geq 0$, if Assumption \ref{ass_variance} hold, we have that a.s.
$$\EE\left[\norm{\epsilon_k}^2|\mc F_k\right]\leq\frac{c\sigma^2}{S_k}.$$
\end{lemma}
\begin{proof}
We first prove that 
$$\EE\left[\norm{\epsilon_k}^2|\mc F_k\right]^{\frac{1}{2}}\leq\frac{c_2\sigma}{\sqrt{S_k}}$$
then the claim follows immediately.
Define the process $\{M_S^S(x)\}_{i=0}^S$ as $M_0(x)=0$ and for $1\leq i\leq S$
$$M_i^S(x)=\frac{1}{S}\sum_{j=1}^iF_\textup{SAA}(x,\xi_j)-\FF(x).$$
Let $\mc F_i=\sigma(\xi_1,\dots,\xi_i)$. Then $\{M_i^S(x),\mc F_i\}_{i=1}^S$ is a martingale starting at $0$.
Let
$$\begin{aligned}
\Delta M_{i-1}^S(x)&=M_i^S(x)-M_{i-1}^S(x)\\
&=F_\textup{SAA}(\bs x,\xi_i)-\FF(\bs x).
\end{aligned}$$
Then, by Equation (\ref{variance}), we have
\begin{equation*}
\EEx{\normsq{\Delta M_{i-1}^S}}^{\frac{1}{2}}=\frac{1}{S}\EEx{\normsq{F_\textup{SAA}(x,\xi_i)-\FF(x)}}^{\frac{1}{2}}\leq \frac{\sigma}{S}.
\end{equation*}
By applying Lemma \ref{BDG_Lemma}, we have
\begin{equation*}
\begin{aligned}
\EEx{\normsq{M_S^S(x)}}^{\frac{1}{2}}&\leq c_2 \sqrt{\sum_{i=1}^{N} \EEx{\normsq{\frac{F_\textup{SAA}(\bs x, \xi_i)-\FF(x)}{S}}}}\\
&\leq c_2 \sqrt{\frac{1}{S^2} \sum_{i=1}^{N} \EEx{\normsq{F_\textup{SAA}(\bs x, \xi_i)-\FF(x)}}}\\
&\leq\frac{c_2\sigma}{\sqrt{S}}.
\end{aligned}
\end{equation*}
We note that $M_{S}^{S}(x^k)=\epsilon_k$, hence by taking the square we conclude that 
\begin{equation*}
\EEk{\normsq{\epsilon_k}}\leq\frac{c\sigma^2}{S_k}. \qedhere
\end{equation*}
\end{proof}
We note that if Lemma \ref{lemma_variance} holds, then it follows that
$$\EE\left[\norm{\Phi^{-1}\varepsilon_k}_\Phi^2|\mc F_k\right]\leq\frac{c\sigma^2\norm{\Phi^{-1}}}{S_k}.$$
We are now ready to prove the main convergence result.

\begin{proof}[Proof of Theorem \ref{theo_SAA_sgne}]
By Lemma \ref{prop_sgne} and \ref{lemma_variance} we have that 
$$\begin{aligned}
\EEk{\normsqphi{\omega^{k+1}-\omega^*}}\leq&\normsqphi{\omega^k-\omega^*}+2\frac{c\sigma^2\norm{\Phi^{-1}}}{S_k}\\
&+\frac{1}{2}\left(\frac{1}{\zeta}-1\right)\op{res}(\omega^k)^2.
\end{aligned}$$
Using Lemma \ref{lemma_RS} 
we conclude that $\{\omega^k\}$ is bounded and has a cluster point $\bar \omega$. Since $\sum\theta_k<0$, $\op{res}(\omega^k)\to0$ as $k\to\infty$ and $\op{res}(\bar \omega)=0$.
\end{proof}

\section{Appendix D\\Proofs of Theorem \ref{theo_SA_sne_coco}}\label{sec_proofs_SNE}

\begin{proof}[Proof of Theorem \ref{theo_SA_sne_coco}]
We start by using nonexpansiveness of the projection:
\begin{equation*}
\begin{aligned}
\normsq{x^{k+1}-x^*}
\leq&\normsq{x^k-x^*}+\gamma^2\normsq{\epsilon_k}+\\
&+2\gamma_k\langle x^k-x^*,\epsilon_k\rangle-\normsq{x^k-x^{k+1}}\\
&-2\gamma^k\langle x^k-x^*,F(x^k)-F(x^*)\rangle+\\
&+2\gamma^k\langle x^k-x^{k+1},\hat F(x^k)-F(x^*)\rangle
\end{aligned}
\end{equation*}
By using cocoercivity of $F$ and the Young's inequality as in \eqref{young} we derive the following upper bound: 
\begin{equation}\label{step}
\begin{aligned}
&\normsq{x^{k+1}-x^*}
\leq\normsq{x^k-x^*}+\zeta\gamma_k^2\normsq{\epsilon_k}\\
&+\left(\frac{1}{\zeta}-1\right)\normsq{x^k-x^{k+1}}+2\gamma_k\langle x^k-x^*,\epsilon_k\rangle\\
&+\left(\frac{\zeta\gamma_k^2}{\beta}-2\gamma^k\right)\langle x^k-x^*,F(x^k)-F(x^*)\rangle\\
&-2\zeta\gamma_k^2\langle F(x^k)-F(x^*),\epsilon_k\rangle
\end{aligned}
\end{equation}
Let the residual be $\op{res}(x^k)=\norm{x^k-\op{proj}(x^k-F(x^k))}$, then
\begin{equation}
\begin{aligned}
\op{res}&(x^k)^2
\leq2\normsq{x^k-x^{k+1}}+\\
&+2\normsq{\op{proj}(x^k-\gamma_k\hat F(x^k,\xi^k))-\op{proj}(x^k-\gamma_kF(x^k))}\\
&\leq2\normsq{x^k-x^{k+1}}+2\normsq{\gamma_k\hat F(x^k,\xi^k))-\gamma_kF(x^k)}\\
&=2\normsq{x^k-x^{k+1}}+2\gamma_k^2\normsq{\epsilon_k}
\end{aligned}
\end{equation}
where the first equality follow by the definition of $x^{k+1}$ and the last inequality follows from non expansivity of the projection. Then,
$$\normsq{x^k-x^{k+1}}\geq\frac{1}{2}\op{res}(x^k)^2-\gamma_k^2\normsq{\epsilon_k}$$
and equation \eqref{step} becomes
\begin{equation*}
\begin{aligned}
\normsq{x^{k+1}-x^*}&\leq\normsq{x^k-x^*}+\left(\zeta-\frac{1}{\zeta}+1\right)\gamma_k^2\normsq{\epsilon_k}+\\
&+\frac{1}{2}\left(\frac{1}{\zeta}-1\right)\op{res}(x^k)^2+2\gamma_k\langle x^k-x^*,\epsilon_k\rangle+\\
&+\left(\frac{\zeta\gamma_k^2}{\beta}-2\gamma^k\right)\langle x^k-x^*,F(x^k)-F(x^*)\rangle+\\
&+2\zeta\gamma_k^2\langle F(x^k)-F(x^*),\epsilon_k\rangle
\end{aligned}
\end{equation*}
By taking the expected value we have that
$$\begin{aligned}
&\EEk{\normsq{x^{k+1}-x^*}}\leq\normsq{x^k-x^*}+\\
&+\left(\zeta-\frac{1}{\zeta}+1\right)\gamma_k^2\EEk{\normsq{\epsilon_k}}+\frac{1}{2}\left(\frac{1}{\zeta}-1\right)\op{res}(x^k)^2.
\end{aligned}$$
Finally, we apply Lemma \ref{lemma_RS} 
to conclude that $(x_k)_{k\in\NN}$ is bounded, therefore it has at least one cluster point $\bar x$. Since $-\frac{1}{2}\op{res}(x_k)^2<0$, by Lemma \ref{lemma_RS} we have that $\op{res}(x_k)\to0$ as $k\to0$. This implies that $\op{res}(\bar x)=0$ that is, $\bar x$ is a SGNE of the game in \eqref{eq_game_stoc}.
\end{proof}

\bibliographystyle{IEEEtran}
\bibliography{Biblio}


\end{document}